\documentclass[12pt]{amsart}%

\usepackage[all]{xy}
\usepackage{amsfonts,amssymb,amsmath,amscd}
\usepackage{amsthm}

\newcounter{zlist}
\newenvironment{zlist}{\begin{list}{{\rm(\arabic{zlist})}}{
\usecounter{zlist}\leftmargin2.5em\labelwidth2em\labelsep0.5em
\topsep0.6ex\itemsep0.3ex plus0.2ex minus0.3ex
\parsep0.3ex plus0.2ex minus0.1ex}}{\end{list}}

\newcounter{blist}
\newenvironment{blist}{\begin{list}{{\rm(\alph{blist})}}{
\usecounter{blist}\leftmargin2.5em\labelwidth2em\labelsep0.5em
\topsep0.6ex \itemsep0.3ex plus0.2ex minus0.3ex
\parsep0.3ex plus0.2ex minus0.1ex}}{\end{list}}

\newcounter{rlist}
\newenvironment{rlist}{\begin{list}{{\rm(\roman{rlist})}}{
\usecounter{rlist}\leftmargin2.5em\labelwidth2em\labelsep0.5em
\topsep0.6ex\itemsep0.3ex plus0.2ex minus0.3ex
\parsep0.3ex plus0.2ex minus0.1ex}}{\end{list}}

\swapnumbers
\newtheorem{theorem}{Theorem}[section]
\newtheorem{lemma}[theorem]{Lemma}
\newtheorem{thm}[theorem]{}
\newtheorem{proposition}[theorem]{Proposition}
\newtheorem{definition}[theorem]{Definition}

\newtheorem{remark}[theorem]{Remark}
\numberwithin{equation}{section}

\newcommand{\bA}{{\mathcal{A}}}
\newcommand{\bB}{{\mathcal{B}}}
\newcommand{\bC}{{\mathcal{C}}}
\newcommand{\bD}{{\mathcal{D}}}
\newcommand{\bF}{{\mathcal{F}}}
\newcommand{\bG}{{\mathcal{G}}}
\newcommand{\bH}{{\mathcal{H}}}
\newcommand{\bT}{{\mathcal{T}}}
\newcommand{\bS}{{\mathcal{S}}}
\newcommand{\bL}{{\mathcal{L}}}
\newcommand{\A}{\mathbb{A}}

\newcommand{\B}{\mathbb{B}}
\newcommand{\D}{\mathbb{D}}

\newcommand{\bP}{\mathcal{P}}
\newcommand{\bR}{\mathcal{R}}
\newcommand{\bQ}{\mathcal{Q}}
\newcommand{\V}{\mathcal{V}}

\newcommand{\FF}{\mathcal{F}^\lambda}
\newcommand{\oF}{F^\lambda}
\newcommand{\m}{m^\lambda}
\newcommand{\e}{e^\lambda}
\newcommand{\z}{\cdot}
\newcommand{\M}{\mathbb{M}}
\newcommand{\bAp}{\mathcal{A}^\tau}

\newcommand{\oK}{\overline{K}}
\newcommand{\oR}{\overline{R}}
\newcommand{\la}{\lambda}
\newcommand{\lra}{\longrightarrow}
\newcommand{\ve}{\varepsilon}
\newcommand{\xra}{\xrightarrow}
\newcommand{\ot}{\otimes}
\newcommand{\Hom}{{\rm Hom}}
\newcommand{\End}{{\rm End}}
\newcommand{\dia}{\diamond}

\newcommand{\LRa}{\Leftrightarrow}
\newcommand{\Ra}{\Rightarrow}

\newcommand{\wt}{\widetilde}
\newcommand{\ol}{\overline}

\newcommand{\dF}{R}
\newcommand{\ev}{{\rm ev}}
\newcommand{\db}{{\rm db}}

\begin{document}

\title{Azumaya monads and comonads}

\author[B. Mesablishvili and R. Wisbauer]{Bachuki Mesablishvili, Tbilisi\\ and \\ Robert Wisbauer, D\"usseldorf}

\begin{abstract} The definition of Azumaya algebras over commutative rings $R$ require
the tensor product of modules over $R$ and the twist map for the tensor product of
any two $R$-modules. Similar constructions are available in braided monoidal categories
and Azumaya algebras were defined in these settings.
Here we introduce Azumaya monads on any category $\A$ by considering a monad $\bF$ on $\A$
endowed with a distributive law $\lambda: FF\to FF$ satisfying the Yang-Baxter
equation (BD-law).
This allows to introduce an {\em opposite monad} $\bF^\la$ and a monad structure on $FF^\la$.
For an {\em Azumaya monad} we impose the condition that the canonical comparison functor induces
an equivalence between the category $\A$ and the category of $\bF\bF^\la$-modules.
Properties and characterisations of these monads are studied, in particular for the
case when $F$ allows for a right adjoint functor.
Dual to Azumaya monads we define {\em Azumaya comonads}
and investigate the interplay between these notions.

In braided categories $(\V,\ot,I,\tau)$, for any $\V$-algebra $A$,
the braiding induces a BD-law $\tau_{A,A}:A\ot A\to A\ot A$ and $A$ is
called left (right) Azumaya, provided the monad $A\ot-$ (resp. $-\ot A$) is Azumaya.
If $\tau$ is a symmetry, or if the category $\V$ admits equalisers and coequalisers,
the notions of left and right Azumaya algebras coincide.
The general theory provides the definition of coalgebras in $\V$.
Given a cocommutative $\V$-coalgebra $\bD$, coalgebras $\bC$
over $\bD$ are defined as coalgebras in the monoidal category of $\bD$-comodules and
we describe when these have the Azumaya property.
In particular, over commutative rings $R$, a coalgebra  $C$ is Azumaya
if and only if the dual $R$-algebra $C^*=\Hom_R(C,R)$ is an Azumaya algebra.
\end{abstract}

\keywords{Azumaya algebras, category equivalences, monoidal categories, (co)monads}

\maketitle

\tableofcontents

\section*{Introduction}

Azumaya algebras $\bA=(A,m,e)$
over a commutative ring $R$ are characterised by the fact that the functor $A\ot_R-$
induces an equivalence between the category of $R$-modules and the category of
$(A,A)$-bimodules.
In this situation Azumaya algebras are separable algebras, that is,
the multiplication $A\ot_R A\to A$ splits as $(A,A)$-bimodule map.

Braided monoidal categories allow for similar constructions as module categories over
commutative rings and so - with some care - Azumaya monoids (algebras) and Brauer groups
can be defined for such categories. For finitely bicomplete categories this was
worked out by J. Fisher-Palmquist in \cite{F-P}, for symmetric monoidal categories
it was investigated by B. Pareigis in \cite{P}, and for braided monoidal categories
the theory was outlined by F. van Oystaeyen and Y. Zhang in \cite{VO-Z} and B. Femi\'c
 in \cite{F}. It follows from the observations in \cite{P} that - even in symmetric
monoidal categories - the category equivalence requested for an Azumaya monoid $A$
does not imply separability of $A$ (defined as for $R$-algebras).

In our approach to Azumaya (co)monads we focus on properties of monads and comonads on
any category $\A$ inducing equivalences between certain related categories.
Our main tools are distributive laws between monads (and comonads)
as used in the investigations of Hopf monads in general categories
(see \cite{MW}, \cite{MW1}).

We begin by recalling basic facts about the related theory - including Galois functors -
in Section \ref{prel}.
Then, in Section \ref{azu}, we consider monads $\bF=(F,m,e)$ on any category
$\A$ endowed with a distributive law
$\lambda: FF\to FF$ satisfying the Yang Baxter equation ({\em BD-laws}).
The latter enables the definition
of a monad $\bF^\lambda =(F^\lambda,m^\lambda,e^\la)$ where
$F^\lambda=F$, $m^\la= m\cdot \lambda$, and
$e^\la=e$. Furthermore, $\lambda$ can be considered as distributive law
$\lambda: F^\la F\to F F^\la$ and this allows to define a monad structure on
$FF^\la$.
Then, for any object $A\in \A$, $F(A)$ allows for an  $\bF\bF^\la$-module structure, thus
inducing a comparison functor $K:\A \to \A_{\bF\bF^\la}$.
We call $\bF$ an {\em Azumaya monad} (in \ref{def-azu})
if this functor is an equivalence of categories.
Properties and characterisations of such monads are given, in particular for the case
that they allow for a right adjoint functor (Theorem \ref{th.1}).

These notions lead to an intrinsic definition of {\em Azumaya comonads} as
 outlined in Section \ref{co-azu} where also
 the relationship between the Azumaya properties of a monad $\bF$ and a right
 adjoint comonad $\bR$ is investigated (Proposition \ref{Galois-Galois}).
It turns out that for a Cauchy complete category $\A$, $\bF$ is an Azumaya monad and
 $FF^\la$ is a separable monad
if and only if $\bR$ is an Azumaya comonad and $G^\kappa G$ is a separable comonad
 (Theorem \ref{Azu-mon-comon}).

In Section \ref{azu-braid}, our theory is applied to study Azumaya algebras in braided
monoidal categories $(\V,\ot,I,\tau)$. Then, for any $\V$-algebra $A$, the braiding
induces a distributive law $\tau_{A,A}:A\ot A\to A\ot A$, and
$A$ is called left (right) Azumaya if the monad $A\ot-:\V\to \V$
(resp. $-\ot A:\V\to \V$) is Azumaya.
In \cite{VO-Z}, $\V$-algebras which are both left and right Azumaya
are used to define the Brauer group of $\V$. We will get various
characterisations for such algebras but will not pursue their role for
the Brauer group. In braided monoidal categories with equalisers and coequalisers,
the notions of left and right Azumaya algebras coincide (Theorem \ref{th.1.Alg.1}).

The results from Section \ref{co-azu} provide
an extensive theory of {\em  Azumaya coalgebras} in braided categories $\V$ and
the basics for this
are described in Section \ref{azu-braid-co}.
Besides the formal transfer of results known for algebras, we introduce coalgebras
$\bC$ over cocommutative coalgebras $\bD$
and for this, Section \ref{co-azu} provides conditions which make them Azumaya.
This extends the corresponding
notions studied for coalgebras over cocommutative coalgebras in vector space
categories by B. Torrecillas, F. van Oystaeyen and Y. Zhang in \cite{Tor-Brauer}.
Over a commutative ring $R$, Azumaya coalgebras $\bC$ turn out to be coseparable
and are characterised by the fact that
the dual algebra $C^*=\Hom(C,R)$ is an Azumaya $R$-algebra.
Notice that coalgebras with the latter property  were first studied by
K. Sugano in \cite{Sugano}.

\section{Preliminaries}\label{prel}

Throughout this section $\A$
will stand for any category.

\begin{thm}\label{mod-comod}{\bf Modules and comodules.} \em
For a monad $\bT=(T, m, e)$ on
$\A$, we write $\A_\bT$ for the Eilenberg-Moore category of $\bT$-modules and
denote the corresponding forgetful-free adjunction by
$$\eta_{\bT}, \varepsilon_{\bT}: \phi_{\bT} \dashv U_{\bT} : \A_{\bT} \to \A.$$

Dually, if $\bG =(G, \delta, \varepsilon)$ is a comonad on $\A$, we write
$\A^\bG$ for the Eilenberg-Moore category of $\bG$-comodules and
denote the corresponding forgetful-cofree adjunction by
$$\eta^{\bG},\varepsilon^{\bG}:U^{\bG} \dashv \phi^{\bG} : \A \to \A^{\bG}.$$

For any monad $\bT=(T, m, e)$ and an adjunction $\overline\eta,\overline\ve:T\dashv R$,
 there is a comonad $\bR=(R,\delta,\ve)$, where $m\dashv \delta$, $\ve\dashv e$ (mates)
and there is an isomorphism of categories (e.g. \cite{MW})
\begin{equation}\label{Psi}
 \Psi:\A_\bT \to \A^\bR,\quad (A,h)\;\mapsto
\;(A,\, A\xra{\overline\eta} RT(A)\xra{R(h)} R(A)).
\end{equation}
Note that, for any $(A,\theta) \in \A^{\bR}$, $\Psi^{-1}(A, \theta)=(A, T(A) \xra{F(\theta)}TR(A) \xra{\overline{\ve}_A} A)$.
\end{thm}

\begin{thm}\label{DL}{\bf Monad distributive laws.} \em
Given two monads $\bT =(T, m, e)$ and $\bS=(S,m', e')$ on $\A$,
a natural transformation $\lambda: TS \to ST$ is
a \textit{(monad) distributive law} of $\bT$ over $\bS$ if it induces
commutativity of the diagrams
$$\xymatrix{
& S \ar[ld]_{e S} \ar[rd]^{S e } &\\
TS \ar[rr]^{\lambda} && ST\\
& T \ar[ru]_{e' T} \ar[lu]^{T e'}&,} \qquad \xymatrix{
TSS \ar[d]_{Tm'} \ar[r]^{\lambda S}& STS \ar[r]^{S\lambda} &SST \ar[d]^{m' T}\\
TS \ar[rr]^{\lambda}&& ST\\
TTS \ar[u]^{m S} \ar[r]_{T \lambda} & TST \ar[r]_{\lambda T}& STT \ar[u]_{Sm}.}$$
Given a distributive law $\lambda: T S\to S T$, the triple
$ \bS\bT=(ST, m' m  \cdot S
\lambda T, e' e)$ is a monad on $\A$ (e.g. \cite{B}, \cite{W-AC}).
Notice that the monad structure on $\bS\bT$ depends on $\la$ and if the choice of $\la$
needs to be specified we write $(\bS\bT)_\la$.

Furthermore, a distributive law $\lambda$ corresponds to a monad
$\widehat{\bS}_\la=(\widehat{S},\widehat{m},\widehat{e})$
on $\A_{\bT}$ that is \emph{lifting} of $\bS$ to $\A_{\bT}$ in the sense that
\begin{center}
$U_{\bT} \widehat{S}=S U_{\bT}$, \;
$U_{\bT} \widehat{m}= m' U_{\bT}$ \; and \; $U_{\bT} \widehat{e}=e' U_{\bT}$.
\end{center}

This defines the Eilenberg-Moore category
$(\A_ {\bT})_{\widehat{\bS}_\la}$ of $\widehat{\bS}_\la$-modules
whose objects are triples $((A,t), s)$, with $(A,t) \in \A_{\bT}$,
$(A,s) \in \A_{{\bS}}$  with a commutative diagram
\begin{equation}\label{mon-mon}
\xymatrix{TS(A) \ar[rr]^{\la_A} \ar[d]_{T(s)}&& ST(A) \ar[d]^{S(t)}\\
T(A)\ar[r]_{t}&A& S(A)\ar[l]^{s}.}
\end{equation}

There is an isomorphism of categories
$\mathcal{P}_\la:\A_{(\bS\bT)_\la} \to (\A_{\bT})_{\widehat{\bS}_\la}$
 by the assignment
$$ (A,S T(A)\xra{ \varrho} A) \; \mapsto \; ((A,\,T(A)\xra{e'_{T(A)}}ST(A) \xra{\varrho}A), S(A) \xra{Se_A }ST(A)\xra{\varrho} A),$$
and for any $((A,t),s) \in (\A_{\bT})_{\widehat{S}_{\la}}$,
$$\mathcal{P}_\la^{-1}((A,t),s)=(A, ST (A) \xra{S(t)} S(A) \xra{s} A).$$
\end{thm}

When no confusion can occur, we shall just write $\widehat{\bS}$ instead of $\widehat{\bS}_\la$.

\begin{thm}\label{la-iso} {\bf Proposition.} In the setting of \ref{DL}, let
$\la:TS\to ST$ be an invertible monad distributive law.

\begin{zlist}
\item  $\la^{-1}: ST\to TS$ is again a monad distributive law;
\item  $\la:TS\to ST$ can be seen as a monad isomorphism
 $(\bT\bS)_{\la^{-1}} \to (\bS\bT)_\la$ defining a category isomorphism
$$\A_\la: \A_{(\bS\bT)_\la}\to \A_{(\bT\bS)_{\la^{-1}}}, \quad
  (A, ST(A)\xra{\,\varrho\,} A) \mapsto  (A, TS(A)\xra{\,\la\,} ST(A)\xra{\,\varrho\,} A) ;$$

\item $\la^{-1}$ induces
 a lifting $\widehat \bT_{\la^{-1}} :\A_\bS\to \A_\bS$ of $\bT$ to $\A_\bS$ and
 an isomorphism of categories
 $$\varPhi: (\A_\bT)_{\widehat \bS_\la }\to (\A_\bS)_{\widehat \bT_{\la^{-1}}},\quad  ((A,t),s) \mapsto ((A,s),t),$$
 leading to the commutative diagram
$$\xymatrix{
  \A_{(\bS\bT)_\la}\ar[r]^{\bP_{\la}} \ar[d]_{\A_\la} &  (\A_\bT)_{\widehat{\bS}_\la}
\ar[d]^{\varPhi} \\
 \A_{(\bT\bS)_{\la^{-1}}} \ar[r]_{\bP_{\la^{-1}}} &  (\A_\bS)_{\widehat {\bT}_{\la^{-1}}} .}
$$
\end{zlist}
\end{thm}
\begin{proof}
 (1), (2) follow by \cite[Lemma 4.2]{KaLaVi}, (3) is outlined in \cite[Remark 3.4]{BS}.
\end{proof}

\begin{thm}\label{co-DL}{\bf Comonad distributive laws.} \em
Given comonads $\bG =(G,\delta,\ve)$ and $\bH=(H,\delta', \ve')$ on $\A$,
 a natural transformation $\kappa: HG \to GH$ is
a \textit{(comonad) distributive law} of $\bG$ over $\bH$ if it induces
commutativity of the diagrams
$$\xymatrix{
& H    &\\
HG \ar[rr]^{\kappa} \ar[ru]^{H\ve}   \ar[rd]_{\ve' G} && GH\ar[lu]_{\ve H} \ar[ld]^{G\ve'}\\
& G &,} \qquad
\xymatrix{
HGG   \ar[r]^{\kappa G}& GHG\ar[r]^{G\kappa} &GGH  \\
HG \ar[rr]^{\kappa} \ar[u]^{H\delta} \ar[d]_{\delta' G} && GH\ar[u]_{\delta H} \ar[d]^{G\delta' }    \\
HHG   \ar[r]_{H \kappa} & HGH \ar[r]_{\kappa H}& GHH  .}$$
Given this, the triple
$(\bH\bG)_\kappa=(HG, H\kappa G \cdot\delta'\delta , \ve' \ve)$ is a comonad on $\A$
(e.g. \cite{B}, \cite{W-AC}).

Also, the distributive law $\kappa$ corresponds to
a lifting  of the comonad $\bH$ to a comonad $\widetilde\bH_\kappa:\A^\bG\to \A^\bG$,
 leading to the Eilenberg-Moore category $(\A^\bG)^{\widetilde{\bH}_\kappa}$ of
$\widetilde\bH_\kappa$-comodules whose objects are triples $((A,g),h)$
with $(A,g)\in \A^\bG$ and $(A,h)\in \A^\bH$ with commutative diagram
$$\xymatrix{ H(A)\ar[d]_{H(g)} & \ar[l]_{\quad h} A \ar[r]^{g\quad} & G(A)\ar[d]^{G(h)} \\
 HG(A) \ar[rr]_{\kappa_A} && GH(A) .}
$$

There is an isomorphism of categories $\bQ_\kappa:\A^{(\bH\bG)_\kappa}\to (\A^\bG)^{\widetilde\bH_\kappa}$
given by
$$(A, A\xra{\rho} HG(A)) \; \mapsto \; (A, A\xra{\,\rho\,} HG(A)\xra{\ve'_{G(A)}} G(A)),
  A \xra{\,\rho\,} HG(A) \xra{H(\ve_A)} H(A)),$$
and for any $((A,g),h)\in (\A^\bG)^{\widetilde\bH_\kappa}$,
$$\bQ_\kappa^{-1}((A,g),h) = (A, A\xra{\, h\, } H(A) \xra{H(g)} HG(A)).$$
\end{thm}

The following observations are dual to \ref{la-iso}.

\begin{thm}\label{ka-iso} {\bf Proposition.} In the setting of \ref{co-DL}, let
$\kappa:HG\to GH$ be an invertible comonad distributive law.

\begin{zlist}
\item  $\kappa^{-1}: GH\to HG$ is again a comonad distributive law of $\bH$ over $\bG$;
\item $GH$ allows for a comonad structure $(\bG \bH)_{\kappa^{-1}}$ and $\kappa:HG\to GH$ is a comonad isomorphism
$ (\bH \bG)_\kappa \to (\bG \bH)_{\kappa^{-1}}$ defining a category equivalence
$$\A^\kappa: \A^{(\bH\bG)_\kappa}\to \A^{(\bG\bH)_{\kappa^{-1}}}, \quad
  (A, A\xra{\,\rho\,} HG(A)) \mapsto  (A, A\xra{\,\rho\,} HG(A)\xra{\,\kappa\,} GH(A) ;$$

\item $\kappa^{-1}$ induces a lifting  $\widetilde \bG_{\kappa^{-1}} :\A^\bH\to \A^\bH$
of $\bG$ to $\A^\bH$ and an equivalence of categories
 $$\varPhi': (\A^\bG)^{\widetilde \bH_\kappa}\to (\A^\bH)^{\widetilde \bG_{\kappa^{-1}}},\quad
  ((A,g),h) \mapsto ((A,h),g),$$
 leading to the commutative diagram
$$\xymatrix{
\A^{(\bH\bG)_\kappa}  \ar[r]^{\bQ_\kappa} \ar[d]_{\A_\kappa} &(\A^\bG)^{\widetilde \bH_\kappa}
         \ar[d]^{\varPhi'} \\
 \A^{(\bG\bH)_{\kappa^{-1}}} \ar[r]_{\bQ_\kappa^{-1}} &  (\A^\bH)^{\widetilde \bG_{\kappa^{-1}}} .}
$$
\end{zlist}
\end{thm}

\begin{thm}\label{MixDL}{\bf Mixed distributive laws.} \em
Given a monad $\bT=(T, m, e)$ and a comonad $\bG=(G,\delta, \varepsilon)$
on $\A$, a \emph{mixed distributive law} (or {\em entwining}) from $\bT$ to
$\bG$ is a
natural transformation $\omega : TG \to GT$ with commutative diagrams
$$
\xymatrix{ & G \ar[dl]_{e  G} \ar[rd]^{G e }& \\
T G \ar[rd]_{T \varepsilon } \ar[rr]_{\omega}&&
G T \ar[ld]^{\varepsilon T}\\
& T& } \qquad
\xymatrix{
 TT G \ar[d]_{m  G} \ar[r]^{T \omega} & T G T \ar[r]^{\omega T} & G T
T \ar[d]^{G m }\\
TG \ar[d]_{T \delta }\ar[rr]_{\omega} && G T \ar[d]^{ \delta  T}\\
TGG \ar[r]_{\omega G} & GTG \ar[r]_{G \omega} & GGT. }
$$

Given a mixed distributive law $\omega : TG \to GT$ from the
monad $\bT$ to the comonad $\bG$,
we write $\widehat{\bG}_\omega=(\widehat{G} ,\widehat{\delta} ,\widehat{\ve} )$
for a comonad on $\A_{\bT}$ lifting $\bG$ (e.g. \cite[Section 5]{W-AC}).

It is well-known that for any object $(A,h)$ of $\A_{\bT}$,
\begin{center}
$ \widehat{G} (A, h)=(G(A), G(h) \cdot \omega_A)$ ,
\;
$(\widehat{\delta} )_{(A, h)}=\delta_A$,
\;
$(\widehat{\varepsilon})_{(A, h)}=\varepsilon_A ,$
\end{center}
and the objects of $(\A_{\bT})^{\widehat{\bG}}$ are triples $(A, h, \vartheta)$, where $(A, h) \in
\A_{\bT}$ and $(A,\vartheta) \in \A^{\bG}$ with commuting diagram
\begin{equation}\label{mix-mod}
\xymatrix{
T(A) \ar[r]^-{h} \ar[d]_-{T(\vartheta)}& A \ar[r]^-{\vartheta}& G(A) \\
TG(A) \ar[rr]_-{\omega_A}&& GT(A). \ar[u]_-{G(h)}}
\end{equation}
\end{thm}

\begin{thm}\label{DisAd}{\bf Distributive laws and adjoint functors.} \em
Let $\lambda: TS \to ST$ be a distributive law of a monad
$\bT=(T,m,e)$ over a monad $\bS=(S,m',e')$ on $\A$.
If $\bT$ admits a right adjoint comonad
$\bR  $ (with $\overline{\eta},\overline{\varepsilon}: T\dashv R $),
then the composite
$$\xymatrix{\lambda_\dia : SR
\ar[r]^-{\overline{\eta} SR  }& R  TSR   \ar[r]^-{R   \lambda R  }& R  STR
\ar[r]^-{R   S \overline{\varepsilon}}& R   S }$$
is a mixed distributive law from  $\bS$ to  $\bR  $ (e.g. \cite{BBW}, \cite{MW})
and the assignment
$$ (A,\, \nu:ST(A) \to A)\;\mapsto \;
(A,\, h_\nu:S(A) \to A,\, \vartheta_\nu: A \to R  (A)), \; \mbox{with}$$
 \begin{center}
$h_\nu:S(A) \xrightarrow{S(e_A)}ST(A) \xrightarrow{\;\nu\;}A$, \;
 $\vartheta_\nu:
 A \xrightarrow{\;\overline{\eta}_A\;}R  T(A) \xrightarrow{R  (e'_{T(A)})}R  ST(A)
\xrightarrow{R  (\nu)}R  (A),$
\end{center}
yields an isomorphism of categories
$\mathbb{A}_{({\bS} {\bT})_\la}\simeq
(\mathbb{A}_{\bS})^{\widehat{\bR}_{\lambda^\dia}}.$
\end{thm}

\begin{thm}\label{InDisAd}{\bf Invertible distributive laws and adjoint functors.} \em
Let $\lambda: TS \to ST$ be an invertible distributive law of a monad
$\bT=(T,m,e)$ over a monad $\bS=(S,m',e')$ on $\A$. Then $\la^{-1}: ST \to TS$ is a distributive law
of the monad $\bS$ over the monad $\bT$ (\ref{la-iso}), and if $\bS$ admits a right adjoint comonad $\mathcal{H}$
(with $\overline{\eta}, \overline{\ve}: S \dashv H$),
then the previous construction can be repeated with $\la$ replaced by $\la^{-1}$. Thus
the composite
$$(\la^{-1})_\dia : TH \xra{\overline{\eta} \,TH} HSTH \xra{H \la^{-1} H} HTSH
\xra{HT \overline{\ve}} HT$$
is a mixed distributive law from the monad $\bS$ to the comonad $\mathcal{H}$.
Moreover, there is an adjunction $\alpha, \beta: \widehat{\bS}_\la \dashv \widehat{\mathcal{H}}_{(\la^{-1})_\dia}:\A_\bT \to \A_\bT,$
where $\widehat\bS_\la$ is the lifting of $\bS$ to $\A_\bT$ considered in \ref{DL}
(e.g. \cite[Theorem 4]{J}) and the canonical isomorphism $\Psi$ from (\ref{Psi})
 yields the commutative diagram
\begin{equation}\label{l-inverse}
\xymatrix{ (\A_\bT)_{\widehat\bS_\la} \ar[r]^{\Psi\quad}
  \ar[d]_{U_{\widehat\bS_\la}} &
 (\A_{\bT})^{\widehat{\mathcal{H}}_{(\la^{-1})_\dia}} \ar[d]^{U^{\widehat{\mathcal{H}}_{(\la^{-1})_\dia}}} \\
   \A_{\bT} \ar[r]_= & \A_{\bT}. }
\end{equation}

Note that $U_\bT(\alpha)=\overline{\eta}$ and $U_\bT(\beta)=\overline{\ve}$.
\end{thm}

\begin{thm}\label{DisAd-E}{\bf Entwinings and adjoint functors.} \em
For a monad $\bT=(T,m,e)$ and a comonad  $\bG=(G,\delta, \varepsilon)$, consider
an entwining  $\omega:TG \to GT$.
If $\bT$ admits a right adjoint comonad
$\bR$ (with $\overline{\eta},\overline{\varepsilon}: T\dashv R $),
then the composite
$$\xymatrix{\omega^\dia : G R
\ar[r]^-{\overline{\eta} G R  }& R  TG R   \ar[r]^-{R   \omega R }& R  G TR
\ar[r]^-{R   G  \overline{\varepsilon}}& R   G  }$$
is a comonad distributive law of $\bG $ over  $\bR  $ (e.g. \cite{BBW}, \cite{MW})
inducing a lifting $\widetilde \bG_\omega$ of $\bG$  to $\A^\bR$ and thus an Eilenberg-Moore
category $(\A^\bR)^{\widetilde \bG_\omega}$ of $\widetilde\bG_\omega$-comodules whose objects are
triples $((A,d),g)$ with commutative diagram
\begin{equation}\label{bi-comod}
\xymatrix {G(A) \ar[d]_{d} & \ar[l]_{\quad g} A \ar[r]^{d\quad } & R (A) \ar[d]^{Rg} \\
GR(A) \ar[rr]^{\omega^\dia_A}  & & R G(A) .}
\end{equation}
\end{thm}
The following notions will be of use for our investigations.

\begin{thm}\label{monadic}{\bf Monadic and comonadic functors.} \em
Let $\eta,\ve:F\dashv R:\B\to \A$ be an adjoint pair of functors.
Then the composite $RF$ allows for a monad structure $\bR \bF$ on $\A$ and
the composite $FR$ for a comonad structure $\bF \bR$ on $\B$.
By definition, $R$ is {\em monadic} and $F$ is {\em comonadic}
provided the respective comparison functors are equivalences,
 $$K_R:\B \to \A_{\bR \bF},\quad B\mapsto (R(B),R(\ve_B)),$$
 $$K_F:\A \to \B^{\bF \bR},\quad A\mapsto (F(A),F(\eta_A)).$$
\end{thm}

For an endofunctor we have, under some conditions on the category:

 \begin{lemma}\label{lem-mon} Let $F:\A\to \A$ be a functor that allows for
a left and a right adjoint functor
and assume $\A$ to have equalisers and coequalisers.
Then the following are equivalent:
\begin{blist}
\item $F$ is conservative;
\item $F$ is monadic;
\item $F$ is comonadic.
\end{blist}
If $\bF=(F,m,e)$ is a monad, then the above are also equivalent to
\begin{blist}
\setcounter{blist}{3}
\item the free functor $\phi_\bF:\A\to \A_\bF$ is comonadic.
\end{blist}
\end{lemma}
\begin{proof}
Since $F$ is a left as well as a right adjoint functor,
it preserves equalisers and coequalisers. Moreover, since $\A$ is assumed to have
both equalisers and coequalisers, it follows from Beck's monadicity theorem
(see \cite{Mc}) and its dual that $F$ is monadic or comonadic if and only if it is
conservative.

 (a)$\LRa$(d) follows from \cite[Corollary 3.12]{Me}.
 \end{proof}

\begin{thm}\label{T-mod}{\bf  $\bT$-module functors.} \em
Given a monad $\bT=(T,m,e)$ on $\A$, a functor $R : \B
\to \A$ is said to be a {\em (left) $\bT$-module} if
there exists a natural transformation $\alpha : TR \to R$ with
$\alpha\cdot eR = 1$ and $\alpha\cdot mR = \alpha \cdot T\alpha$.

This structure of a left  $\bT$-module on $R$ is equivalent to the existence of a functor
$\overline{R} : \B \to \A_\bT$ with commutative diagram
(see \cite[Proposition II.1.1]{D})
$$\xymatrix{ \B \ar[r]^\oR \ar[dr]_R & \A_\bT \ar[d]^{U_\bT} \\
           & \A.}$$

If $\overline{R}$ is such
a functor, then $\overline{R}(B)=(R(B), \alpha_{B})$ for some morphism
$\alpha_{B}: TR(B) \to R(B)$ and the collection $\{\alpha_B,\, B \in \B\}$
forms a natural transformation $\alpha:TR \to R$  making $R$ a $\bT$-module.
Conversely, if $(R,\alpha: TR \to R)$ is a $\bT$-module, then $\overline{R} : \B
\to \A_\bT$ is defined by $\overline{R}(B)=(R(B), \alpha_B)$.

 For any $\bT$-module
$(R: \B \to\A,\alpha)$ admitting an adjunction  $F\dashv R : \B \to \A$ with
unit $\eta:1\to RF$, the
composite
$$ t_{\oR}: \xymatrix{T \ar[r]^-{T \eta}& TRF \ar[r]^-{\alpha F} & RF}$$
is a monad morphism from $\bT$ to the monad $\bR \bF$ on $\A$
generated by the adjunction $F \dashv R$. This yields a functor
$\A_{t_{\overline{R}}}: \A_{\bR \bF} \to \A_\bT$

If $t_{\overline{R}}: T \to RF$ is an isomorphism (i.e. $\A_{t_{\overline{R}}}$
is an isomorphism), then $R$ is called a $\bT$-\emph{Galois module functor}.
Since $\overline{R} =  \A_{t_{\overline{R}}}\cdot K_R$ (see \ref{monadic}) we have
(dual to \cite[Theorem 4.4]{M}):
\end{thm}
\begin{proposition}\label{P.1.2}
The functor $\overline{R}$ is an equivalence of categories if and
only if the functor $R$ is monadic and a $\bT$-Galois module functor.
\end{proposition}

 \begin{thm}\label{Gal}{\bf $\bG$-comodule functors.} \em
Given a comonad $\bG=(G,\delta,\varepsilon)$ on a category
$\A$, a functor $L: \B \to \A$ is a {\em left $\bG$-functor}
if there exists a natural transformation $\alpha: L \to GL$ with
$\ve L\cdot \alpha=1$ and $\delta L\cdot \alpha = G\alpha \cdot \alpha$.
This structure on $L$ is equivalent to
the existence of a functor $\overline{L} : \B \to \A^\bG$
with commutative diagram (dual to \ref{T-mod})
$$\xymatrix{ \B \ar[r]^{\overline{L}} \ar[dr]_L & \A^{\bG} \ar[d]^{U^{\bG}}\\
& \A .}$$

If a $\bG$-functor $(L,\alpha)$ admits a right adjoint $S: \A \to \B$, with
counit $\sigma : LS \to 1$, then
(see  Propositions II.1.1 and II.1.4 in \cite {D}) the composite
$$\xymatrix{t_{\overline{L}}:LS \ar[r]^-{\alpha S}& GLS
\ar[r]^-{G \sigma }  & G}$$
is a comonad morphism from the comonad
generated by the adjunction $L \dashv S$ to $\bG$.

 $L: \B \to \A$ is said to be a $\bG$-\emph{Galois comodule functor}
 provided $t_{\overline{L}}: LS \to G$ is an isomorphism.
\end{thm}

Dual to Proposition \ref{P.1.2} we have
(see also \cite{MW-G}, \cite{MW1}):

 \begin{proposition}\label{P.1.7}
The functor $\overline{L}$ is an equivalence of categories
if and only if the functor $L$ is comonadic and a $\bG$-Galois comodule functor.
\end{proposition}

\begin{thm}\label{right-adj}{\bf Right adjoint for $\overline{L}$.} \em
If the category $\B$ has equalisers of coreflexive pairs and $L\dashv S$, the
functor $\overline{L}$ (in \ref{Gal}) has a right adjoint $\overline{S}$,
which can be described as follows (e.g. \cite{D}, \cite{M}): With the composite
$$\gamma:\xymatrix{S
\ar[r]^-{\eta S}& SLS \ar[r]^{S t_{\overline{L}}}& SG,}$$
the value of $\overline{S}$ at $(A, \vartheta) \in \A^{\bG}$
is given by the equaliser
$$\xymatrix{\overline{S}(A, \vartheta)
\ar[r]^-{i_{(A, \vartheta)}}&S(A) \ar@{->}@<0.5ex>[rr]^-{S(\vartheta)}
\ar@{->}@<-0.5ex>[rr]_-{\gamma_A} & & SG(A).}$$

If $\overline{\sigma}$ denotes the counit of the adjunction
$\overline{L} \dashv \overline{S}$,
then for any $(A, \vartheta)\in \A^{\bG}$,
\begin{equation}\label{counit}
U^\bG(\overline{\sigma}_{(A, \vartheta)})=\sigma_{A} \cdot L(i_{(A,\vartheta)})\,,
\end{equation}
where $\sigma: LS \to 1$ is the counit of the adjunction $L \dashv S$.
\end{thm}

\begin{thm}\label{sep-func}{\bf Separable functors.} (e.g. \cite{R})
\em A functor $F: \A \to \B$ between any categories is said to be \emph{separable}
if the natural transformation
$$F_{-,-}: \A(-,-) \to \B(U(-),U(-))$$ is a split monomorphism.
\smallskip

{\em If $F:\A\to \B$ and $G:\B\to \D$ are functors, then
\begin{rlist}
\item if $F$ and $G$ are separable, then $GF$ is also separable;
\item if $GF$ is separable, then $F$ is separable.
\end{rlist}
}
\end{thm}

\begin{thm}\label{sep-mon}{\bf Separable (co)monads.} (\cite[2.9]{BBW})
Let $\A$ be any category.
\begin{zlist}
\item For a monad $\bF=(F,m,e)$ on $\A$, the following are equivalent:
\begin{blist}
\item $m$ has a natural section $\omega:F \to FF$ such that
$Fm\cdot\omega F=\omega\cdot m=mF\cdot F\omega$;
\item the forgetful functor $U_\bF:\A_\bF\to \A$ is separable.
\end{blist}
\item For a comonad $\bG=(G,\delta,\ve)$ on $\A$, the following are equivalent:
\begin{blist}
\item $\delta$ has a natural retraction $\varrho:GG \to G$ such that
$\varrho G\cdot G\delta = \delta\cdot \varrho = G\varrho\cdot \delta G$;
\item the forgetful functor $U^\bG: \A^\bG\to \A$ is separable.
\end{blist}
\end{zlist}
\end{thm}

\begin{thm}\label{adj-sep}
{\bf Separability of adjoints.} (\cite[2.10]{BBW})
Let $G:\A\to \A$ and $F:\A\to \A$ be an adjoint pair of functors
with unit $\bar\eta:1_\A \to FG$ and counit $\bar\ve:GF\to 1_\A$.
\begin{zlist}
\item $F$ is separable if and only if $\bar\eta:1_\A \to FG$ is a split monomorphism;
\item $G$ is separable if and only if $\bar\ve:GF\to 1_\A$ a split epimorphism.
\end{zlist}
Given a comonad structure $\bG$ on $G$ with corresponding monad structure $\bF$ on $F$ (see \ref{mod-comod}),
there are pairs of adjoint functors
 $$\begin{array}{rl}
\xymatrix{\A \ar[r]^{\phi_\bF} & \A_\bF},
\xymatrix{ \A_\bF \ar[r]^{U_\bF}& \A },& \quad
\xymatrix{\A^\bG \ar[r]^{U^\bG}  & \A},
\xymatrix{\A \ar[r]^{\phi^\bG}  & \A^\bG },
\end{array}$$
\begin{zlist}
\item
$\phi^\bG$ is separable if and only if $\phi_\bF$ is separable.
\item
$U^\bG$ is separable if and only if $U_\bF$ is separable
and then any object of ${\mathbb A}^G$ is injective relative to $U^G$ and every object
of ${\mathbb A}_\bF$ is projective relative to $U_\bF$.
\end{zlist}
\end{thm}

The following generalises  criterions for separability given in  \cite[Theorem 1.2]{R}.

\begin{proposition}\label{rafael} Let
$U:\A\to \B$ and $F:\B\to\A$ be a pair of functors.
\begin{rlist}
\item
  If there exist natural transformations $1 \xra{\kappa} FU \xra{\kappa'}1$
such that $\kappa' \cdot \kappa=1$, then both $FU$ and $U$ are separable.
\item
  If there exist natural transformations $1 \xra{\eta} UF \xra{\eta'}1$
such that $\eta' \cdot \eta=1$, then both $UF$ and $F$ are separable.
\end{rlist}
\end{proposition}
\begin{proof} (i)  Inspection shows that
$$ \A(-,-) \xra{(FU)_{-,-}} \A(FU(-),FU(-))\xra{\A(\kappa,\kappa')}\A(-,-)$$
is the identify and hence $FU$ is separable. By \ref{sep-func}, this implies that $U$ is also separable.

(ii) is shown symmetrically.
\end{proof}

\section{Azumaya monads}\label{azu}

An algebra $A$ over a commutative ring $R$ is Azumaya provided $A$ induces
an equivalence between $\M_R$ and the category $_A\M_A$ of $(A,A)$-bimodules.
The construction uses properties of the monad $A\ot_R-$ on $\M_R$
and the purpose of this section is to trace this notion back to the categorical
essentials to allow the formulation of the basic properties for monads on any category.
Throughout again $\A$ will denote any category.

\begin{thm}{\bf Definitions.}\label{def-BD} \em
Given an endofunctor $F:\A\to \A$ on $\A$, a natural transformation $\lambda:FF\to FF$
is said to satisfy the {\em Yang-Baxter equation}
provided it induces commutativity of the diagram
$$
\xymatrix{
FF F \ar[r]^{ F\la } \ar[d]_{\la F}&F F F \ar[r]^{\la F} & F F F \ar[d]^{F \la}\\
F F F  \ar[r]^{F\la}&FFF\ar[r]^{\la F}& F F F  . }
$$

For a monad $\bF=(F,m,e)$ on $\A$, a monad distributive law $\lambda:FF\to FF$
satisfying the Yang-Baxter equation
is called a {\em (monad) BD-law} (see \cite[Definition 2.2]{KaLaVi}).
\end{thm}

Here the interest in the YB-condition for distributive laws lies in the fact that
it allows to define {\em opposite} monads and comonads

\begin{proposition}\label{prop.1}
Let $\bF=(F, m, e)$ be a monad on $\A$ and $\la : FF \to FF$ a BD-law.
\begin{zlist}
\item   $\FF=(\oF,\m , \e)$ is a monad on $\A$, where $\oF=F$, $\m=m \z \la$ and $\e=e$.

\item $\la$ defines a distributive law $\la:\bF^\la \bF \to \bF \bF^\la$
making $ \bF \bF^{\la} =(FF , \underline m, \underline e)$
a monad   where
$$\underline m = mm^\la \z F\la F: FF FF  \to FF, \quad
   \underline e:= ee: 1\to FF.$$

\item The composite $FFF \xra{F\la } FFF \xra{Fm} FF \xra{m} F$ defines a left $ \bF \bF^{\la}$-module
structure on the functor $F:\A \to \A$.

\item There is a comparison functor   $\overline{K}_\bF: \A \to \A_{\bF\bF^\la}$ given by
$$   A \; \mapsto \;
(F(A),\, FF F(A)\xra{F(\la_A)} FF F (A)\xra{F( m_{A})} FF(A) \xra{m_A} F(A)).
 $$
\end{zlist}
\end{proposition}
\begin{proof}
(1) is easily verified (e.g. \cite[Remark 3.4]{BS}, \cite[Section 6.9]{MW}).

(2) can be seen by direct computation (e.g. \cite{BS}, \cite{KaLaVi}, and \cite{MW}).

(3) can be proved by a straightforward diagram chase.

(4) follows from \ref{T-mod} using the left $ \bF \bF^{\la}$-module structure of $F$ defined in (3).
\end{proof}

When no confusion can occur, we shall just write $\overline{K}$ instead of $\overline{K}_\bF$.

\begin{definition}\label{def-azu} \em
A monad $\bF=(F, m, e)$ on any category $\A$ is said to be \emph{Azumaya}
provided it allows for a BD-law $\la : FF \to FF$
such that the comparison functor
$\oK_\bF: \A\to \A_{\bF\bF^\la}$ is an equivalence of categories.
\end{definition}

\begin{proposition}\label{l-r-adj} If $\bF$ is an Azumaya monad on  $\A$,
then the functor $F$ admits a left adjoint.
\end{proposition}
\begin{proof}  With our previous notation we have the commutative diagram
\begin{equation}\label{Azu.Mon}
\xymatrix{ \A \ar[rr]^-{\overline{K}_\bF} \ar[drr]_{F}
 & & \A_{\bF\bF^\la} \ar[d]^{U_{\bF\bF^\la}}\\
&&\A \, .}
\end{equation}

\noindent Since $U_{\bF\bF^\la}:\A_{\bF\bF^\la} \to \A$ always has a left adjoint,
and since $\overline{K}_\bF$ is an equivalence of categories, the composite
$F=U_{\bF\bF^\la} \z\overline{ K}_\bF$ has a left adjoint.
\end{proof}

This observation allows for a first characterisation of Azumaya monads.

\begin{theorem} \label{th.1.Monad} Let $\bF=(F, m, e)$ be a monad on $\A$,
 $\la : FF \to  FF$ a BD-law. The following are equivalent:
\begin{blist}
\item $\bF$ is an Azumaya monad;

\item the functor $F:\A \to \A$ is monadic
and the left $\bF \bF^\la$-module structure on $F$ defined in Proposition \ref{prop.1} is Galois;
\item the functor $F:\A \to \A$ is monadic (with some adjunction $\eta, \ve: L\dashv F$)
and the composite (as in \ref{T-mod})
$$t_{\oK}: F F \xra{FF \eta} FF FL \xra{ F\la L} FF FL
\xra{FmL} FFL \xra{mL} FL $$
is an isomorphism of monads $\bF \bF^\la \to \bT$, where $\bT$ is the monad on
$\A$ generated by this adjunction $L\dashv F$.
\end{blist}
\end{theorem}
\begin{proof} That (a) and (b) are equivalent follows from Proposition \ref{P.1.7}.

(b)$\LRa$(c) In both cases, $F$ is monadic and thus $F$ allows for an adjunction,
say  $L\dashv F$ with unit $\eta:1\to FL$. Write $\bT$ for the monad on $\A$ generated
by this adjunction. Since the left $\bF\bF^\la$-module structure on the functor $F$
is the composite
$$FF F \xra{F\la } FFF \xra{Fm} FF \xra{m} F,$$
it follows from \ref{T-mod} that the monad morphism $t_{\oK}: \bF\bF^\la \to \bT $
induced by the diagram
$$
\xymatrix{ \A \ar[rr]^-{\overline{K}_\bF} \ar@/_1pc/ [ddr]_{F}
 & & \A_{\bF\bF^\la} \ar@/_1pc/[ddl]|{U_{\bF\bF^\la}}\\\\
&\A \ar@/_1pc/[luu]^L \ar@/_1pc/[ruu]|{\phi_{\bF\bF^\la}} &}
$$ is the composite
$$t_{\oK}: F F \xra{FF \eta} FF FL \xra{ F \la L} FF FL
\xra{FmL} FFL \xra{mL} FL .$$
Thus $F$ is $\bF\bF^\la$-Galois if and only if $t_{\oK}$
is an isomorphism.
\end{proof}

\begin{thm}\label{iso-P}
{\bf The isomorphism $\A_{\bF \bF^{\la}} \simeq (\A_{\bF^\la})_{\widehat{\bF}}$.} \em
According to \ref{DL}, for any BD-law $\la :  F F \to F F$, the assignment
$$ (A, FF(A) \xra{\varrho} A) \; \mapsto \; ((A,F(A)\xra{e_{F(A)}}FF(A) \xra{\varrho}A),
F(A) \xra{Fe_A }FF(A)\xra{\varrho} A)$$
yields an isomorphism of categories \;
$\mathcal{P}_\la:\A_{ \bF \bF^{\la}} \lra (\A_{\bF^\la})_{\widehat{\bF}},$
where for any $((A,h),g) \in (\A_{\bF^\la})_{\widehat{\bF}}$,
$$\mathcal{P}_\la^{-1}((A,h),g)=(A, FF (A) \xra{F h} F(A) \xra{g} A).$$

{\em There is a comparison functor \;  $K=K_\bF:\A \to (\A_{\FF})_{\widehat{\bF}}$,
$$  A \; \mapsto \; ((F(A), F F (A) \xra{\la_A} FF(A) \xra{m_A} F(A)),
  F F(A) \xra{m_A} F(A)),\\
$$
with $\oK_\bF =\bP_\la^{-1}K_\bF$ and commutative diagram
$$
\xymatrix{ \A \ar[r]^-{K_\bF} \ar[dr]_{\phi_{\bF^\la}} &
 (\A_{\bF^\la})_{\widehat{\bF}}\ar[d]^{U_{\widehat{\bF}}}
  \ar[r]^{\bP_\la^{-1}} & \A_{\bF\bF^\la} \ar[d]^{U_{\bF\bF^\la}}  \\
& \A_{\bF^\la} \ar[r]_{U_{\bF^\la} } & \A \, . }
$$
}
\end{thm}

\begin{proof} Direct calculation shows that
$$\mathcal{P}_\la\overline{K}_\bF(A)=((F(A), F F (A) \xra{\la_A} FF(A) \xra{m_A} F(A)),
  F F(A) \xra{m_A} F(A)),$$
for all $A \in \A$.
\end{proof}

It is obvious that $ \overline{K}_\bF: \A\to \A_{\bF\bF^\la}$ is an equivalence
(i.e. $\bF$ is Azumaya)
if and only if $K_\bF:\A\to  (\A_{\bF^\la})_{\widehat{\bF}}$ is an equivalence. To apply
Proposition \ref{P.1.2} to the functor $K_\bF$, we will need a functor left adjoint to
$\phi_{\bF^\la}$
whose existence is not a consequence of the Azumaya condition.
For this the invertibility of $\la$ plays a crucial part.

\begin{proposition}\label{iso-P-in}
Let $\bF=(F, m, e)$ be a monad on $\A$ with an invertible BD-law $\la : FF \to FF$.
\begin{zlist}
\item
   $\lambda^{-1}: \bF \bF^\la\to \bF^\la \bF$ is a distributive law inducing a monad
$(\bF^\la \bF)_{\la^{-1}}=(F F, \underline{\underline m}, \underline{\underline e})$
   where
$$  \underline{\underline m} = m^\la m \z F\la^{-1} F: F FF F \to F F, \quad
 \underline{\underline e} = e e: 1\to  F F,$$
and  $\la$ is an isomorphism of monads $ (\bF^\la \bF)_{\la^{-1}} \to (\bF \bF^\la)_\la$.

\item There is an isomorphism of categories
 $$ \varPhi:(\A_{\bF^\la})_{\widehat{\bF}_\la}  \to (\A_{\bF})_{(\widehat{\bF^\la})_{\la^{-1}}}, \quad
  ((A,h),g) \mapsto ((A,g),h).$$
\item $\la^{-1}$
induces a comparison functor \;
$K'_\bF: \A \to (\A_{\bF})_{(\widehat{\bF^\la})_{\la^{-1}}} \; ( \simeq  \A_{(\bF^\la \bF)_{\la^{-1}}})$,
$$A \; \mapsto \; ((F(A),\, FF(A)\xra{m_A} F(A)), F F(A)\xra{\la_A} FF(A) \xra{m_A}F(A)),$$
with  commutative diagrams
$$
\xymatrix{ \A \ar[r]^-{K'_\bF}\ar[dr]_{\phi_{\bF}} & (\A_{\bF})_{(\widehat{\bF^\la})_{\la^{-1}}}
\ar[d]^{U_{(\widehat{\bF^\la})_{\la^{-1}}}}\\
& \A_{\bF}; }\qquad
\xymatrix{ \A \ar[r]^-{K_\bF} \ar[dr]_{K'_\bF} & (\A_{\bF^\la})_{\widehat{\bF}_\la}
 \ar[d]^{\varPhi}\\
&    (\A_{\bF})_{(\widehat{\bF^\la})_{\la^{-1}}} .   }
$$
\end{zlist}
\end{proposition}
\begin{proof}
  (1), (2)  follow by Proposition \ref{la-iso},
  (3)  is shown similarly to \ref{iso-P}.
\end{proof}

For $\la$ invertible, it follows from the diagrams in the
Sections \ref{iso-P}, \ref{iso-P-in}
that $F$ is an Azumaya monad if and only if the functor
$$K'_\bF: \A \to (\A_{\bF})_{(\widehat{\bF^\la})_{\la^{-1}}}$$ is an equivalence of categories.

Note that if $\la: \bF\bF \to \bF\bF$ is a BD-law, then $\la$ can be seen as a
BD-law  $\la : \bF^\la \bF^\la \to \bF^\la \bF^\la$, and it is not hard to see
that the corresponding comparison functor $$K_{\bF^\la}:\A \to (\A_{(\bF^\la)^\la})_{(\widehat{\bF^\la})_\la}$$
takes $A \in \A$ to $$(F(A),\, FF F(A)\xra{F(\la_A)} FF F (A)\xra{F( (m^\la)_{A})} FF(A) \xra{(m^\la)_{A}} F(A)).$$

Now, if $\la^2=1$, then $\la=\la^{-1}$ and $(\bF^\la)^\la=\bF$. Thus, the category $(\A_{(\bF^\la)^\la})_{(\widehat{\bF^\la})_\la}$
can be identified with the category $(\A_{\bF})_{(\widehat{\bF^\la})_{\la^{-1}}}$.
Modulo this identification, the functor $K'_{\bF^\la}$ corresponds to the functor $K_{\bF^\la}$.
It now follows from the preceding remark:

\begin{proposition}\label{sq}Let $\bF=(F, m, e)$ be a monad on $\A$ with a BD-law $\la : FF \to FF$.
If $\la^2=1$, then the monad $\bF$ is Azumaya if and only if the monad $\bF^\la$ is so.
\end{proposition}

\begin{thm}\label{Azu-ra} {\bf Azumaya monads with right adjoints.} \em
Let $\bF=(F,m,e)$ be a monad with an invertible BD-law  $\la:FF\to FF$. Assume $F$ to admit a right adjoint functor $R$,
with $\overline{\eta}, \overline{\varepsilon}: F\dashv \dF $,
inducing a comonad $\bR=(R,\delta,\ve)$ (see \ref{mod-comod}).
  Since  $\la:\bF^\la \bF\to \bF\bF^\la$ is an invertible distributive law,
there is a comonad $\widehat{\bR}=\widehat{\bR}_{(\la^{-1})_\diamond}$ on $\A_{\bF^\la}$ lifting the comonad $\bR$
 and is right adjoint
to the monad $\widehat{\bF}$ (see \ref{DisAd}) yielding a category isomorphism
$$
\Psi_{\bF^\la}:(\A_{\bF^\la})_{\widehat{\bF}_\la} \to (\A_{\bF^\la})^{\widehat{\bR}},\quad
$$
where for any $((A,h),g) \in (\A_{\bF^\la})_{\widehat{\bF}_\la}$,
\begin{center}
$\Psi_{\bF^\la}((A,h),g)=((A,h), \widetilde{g})$ \;
with \quad $\widetilde{g}:A \xra{\overline{\eta}_A} \dF  F(A) \xra{\dF  (g)}\dF  (A),$
\end{center}
and a commutative diagram (see (\ref{l-inverse}))
\begin{equation}\label{Azum.1b}
\xymatrix{ \A \ar[r]^-{K} \ar[dr]_{\phi_{\bF^\la}} &
  (\A_{\bF^\la})_{\widehat{\bF}_\la} \ar[d]^{U_{\widehat{\bF}_\la}}
        \ar[r]^{{\Psi}_{\bF^\la}} &
 (\A_{\bF^\la})^{\widehat{\bR}} \ar[d]^{U^{\widehat{\bR}}}\\
& \A_{\bF^\la} \ar[r]^= & \A_{\bF^\la}  \,.}
\end{equation}

Putting
$\underline{K}:=\A \xra{K}(\A_{\bF^\la})_{\widehat{\bF}_\la}
  \xra{\Psi_{\bF^\la}}(\A_{\bF^\la})^{\widehat{\bR}},$
one has for any $A\in  \A$,
\begin{center}
$\underline{K}(A)=((F(A),m_A \z \la_A ),\dF  (m_A )\z \overline{\eta}_{F(A)})$.
\end{center} So the  $A$-component $\alpha_A$ of the induced $\widehat{\bR}$-comodule
structure $\alpha:\phi_{\bF^\la} \to \widehat{\bR}\phi_{\bF^\la}$
on the functor $\phi_{\bF^\la}$ induced by the commutative diagram (\ref{Azum.1b}) (see \ref{Gal}), is the composite
$$\alpha_A: F(A) \xra{\overline{\eta}_{F(A)}}\dF  FF(A) \xra{\dF  (m_A )}\dF  F(A).$$
It then follows that for any $(A,h) \in \A_{\bF^\la}$, the $(A,h)$-component $t_{(A,h)}$ of the corresponding
comonad morphism  $t:\phi_{\bF^\la}U_{\bF^\la}\to \widehat{\bR}$   (see \ref{Gal})
 is the composite
\begin{equation}\label{t-comp}
t_{(A,h)}:F(A) \xra{\overline{\eta}_{F(A)}}\dF  FF(A) \xra{\dF  (m_A)}\dF  F(A)
\xra{\dF  (h)}\dF  (A).
\end{equation}
\end{thm}

These observations lead to the following characterisations of Azumaya monads.

\begin{theorem} \label{th.1} Let $\bF=(F, m, e)$ be a monad on $\A$,
 $\la : FF \to  FF$ an invertible  BD-law,  and $\bR$ a comonad right adjoint to  $\bF$
(with $\overline{\eta}, \overline{\varepsilon}: F\dashv R $).
Then the following are equivalent:
\begin{blist}
\item $\bF$ is an Azumaya monad;
\item
\begin{rlist}
\item $\phi_{\bF^\la}$ is comonadic and
\item $\phi_{\bF^\la}$ is $\widehat{\bR}$-Galois, that is,

$t_{(A,h)}$ in (\ref{t-comp})
is an isomorphism   for any $(A,h) \in \A_{\bF^\la}$, or the composite\\
$\chi: FF \xra{\overline{\eta} FF} \dF  FFF
\xra{\dF mF} \dF  FF
\xra{\dF  \la} \dF  FF \xra{\dF m} \dF  F$
is an isomorphism.
\end{rlist}
\end{blist}
\end{theorem}
\begin{proof} Recall first that the monad $\FF$ is of effective descent type means that
$\phi_{\bF^\la}$ is comonadic.

By Proposition \ref{P.1.7}, the functor $\underline{K}$ making the triangle
(\ref{Azum.1b}) commute
is an equivalence of categories (i.e., the monad $\bF$ is Azumaya) if and only if the monad
$\FF$ is of effective descent type and the comonad morphism $t:\phi_{\bF^\la}U_{\bF^\la}\to \widehat{\bR}$
is an isomorphism. Moreover, according to  \cite[Theorem 2.12]{MW1}, $t$ is an isomorphism if and only if
for any object $A \in \A$, the $\phi_{\FF}(A)$-component
$t_{\phi_{\FF}(A)}:F \phi_{\FF}(A) \to \dF  \phi_{\FF}(A)$ is an isomorphism.
Using now that $\phi_{\bF^\lambda}(A)=(F(A), \m_A =m_A  \z \la_A)$, it is easy to
see that $\chi_A=t_{\phi_{\FF}(A)}$ for all $A \in \A$. This completes the proof.

\end{proof}

The existence of a right adjoint of the comparison functor $\underline{K}$
can be guaranteed by conditions on the base category.

\begin{thm}\label{central} {\bf Right adjoint for $\underline{K}$.}
With the data given above, assume
 $\A$ to have  equalisers of coreflexive pairs. Then
\begin{zlist}
\item the functor
$\underline{K}:\A \to (\A_{\bF^\la})^{\widehat{\bR}}$ (see \ref{Azu-ra}) admits a right adjoint
$\underline{R}:(\A_{\bF^\la})^{\widehat{\bR}}\to \A$
whose value at $((A,h),\vartheta)\in (\A_{\bF^\la})^{\widehat{\bR}}$
is the equaliser
$$
\xymatrix{
\underline{R}((A,h),\vartheta) \ar[rr]^{\qquad i_{((A,h),\vartheta)}} &&
A \ar@/^1pc/@{->}[rr]^{\vartheta} \ar[rd]_{\overline{\eta}_A} &
 & R(A) \\
 &&& RF(A)\ar[ru]_{R(h)}& ;}
 $$

\item  for any
$A\in \A$, $\underline{R} \,\underline{K}(A)$
is the equaliser
$$
\xymatrix{
\underline{R} \,\underline{K}(A) \ar[rr]^{i_{\ol{K}(A)}} &&
F(A) \ar@/^1pc/@{->}[rrrr]^{R(m_A)\cdot\, \overline{\eta}_{ F(A)}}
 \ar[rrd]_{R(\la_A) \cdot \,\overline{\eta}_{ F(A)}} &&
 && RF(A) \\
&&&& RFF(A) \ar[rru]_{R(m_A)}& .}
 $$
\end{zlist}
\end{thm}
\begin{proof} (1) According to \ref{right-adj}, $\overline{R}((A,h),\vartheta)$ is the
object part of the equaliser of
$$\xymatrix{A \ar@{->}@<0.5ex>[rr]^-{ \vartheta }
\ar@ {->}@<-0.5ex> [rr]_-{  \gamma_{(A,h)} }& & R(A)\,,}$$
where $\gamma$ is the
composite $U_{\bF^\la} \xra{U_{\bF^\la}e} U_{\bF^\la}\phi_{\bF^\la}U_{\bF^\la}=U_{\bF^\la}F\xra{U_{\bF^\la}t}U_{\bF^\la}\widehat{R}.$
It follows from the description of $t$ that $\gamma_{(A,h)}$ is the composite
$$A \xra{\;e_A \; } F(A)\xra{\overline{\eta}_{F(A)}}\dF  FF(A)
\xra{\dF  (m_A)}\dF  F(A) \xra{\dF  (h)}\dF  (A)$$
which is just the composite $\dF  (h) \z \overline{\eta}_{A}$
since
\begin{itemize}
  \item $ \overline{\eta}_{F(A)} \z e_A= \dF  F(e_A) \z \overline{\eta}_{A}$ by naturality of $\overline{\eta}$, and
  \item $m_A \z F(e_A)=1$ because $e$ is the unit for $\bF$.
\end{itemize}

(2) For any $A \in \A$, $\underline{K}(A)$ fits into the diagram (\ref{Azum.1b}).
\end{proof}

\begin{thm}{\bf Definition.} \em
Write $F_F$ for the subfunctor of the functor $F$ determined by the equaliser of the
diagram
$$
\xymatrix{
F \ar@/^1pc/@{->}[rrrr]^{Rm\cdot \, \overline{\eta} F} \ar[rrd]_{R\la \cdot \,\overline{\eta}F} &&
 && RF \\
&& RFF \ar[rru]_{Rm}& & . }
 $$
We call the monad $\bF$ \emph{central} if $F_F$ is (isomorphic to) the identity functor.
\end{thm}
Since $\underline{R}$ is right adjoint to the functor $\underline{K}$, $\underline{K}$
is fully faithful if and only if $\underline{R} \,\underline{K}\simeq 1$.

\begin{theorem}\label{th.2} Assume $\A$ to admit equalisers of coreflexive pairs.
Let $\bF=(F, m, e)$ be a monad on  $\A$, $\la: FF \to FF$ an invertible  BD-law,
and $\bR$ a comonad right adjoint to $\bF$.
Then the comparison functor $\underline{K}:\A \to (\A_{\bF^\la})^{\widehat{\bR}}$ is
\begin{rlist}
 \item full and faithful if and only if the monad $\bF$ is central;
\item an equivalence of categories
      if and only if the monad $\bF$ is central and the functor $\underline{R}$ is conservative.
\end{rlist}
\end{theorem}

\begin{proof} (i)  follows from the preceding proposition.

(ii) Since $\bF$ is central, the unit
$\underline{\eta}:1 \to \underline{R} \,\underline{K}$ of the adjunction
$\underline{K} \dashv \underline{R}$ is an isomorphism by (i).
If  $\underline{\varepsilon}$ is the counit
of the adjunction, then it follows from the triangular identity
$\underline{R}\,\underline{\varepsilon}\cdot \underline{\eta}\,\underline{R}=1$ that
$\underline{R}\,\underline{\varepsilon}$ is an isomorphism. Since $\underline{R}$ is assumed to be conservative
(reflects isomorphisms), this implies that $\underline{\varepsilon}$ is an isomorphism, too.
Thus $\underline{K}$ is an equivalence of categories.
\end{proof}

\section{Azumaya comonads}\label{co-azu}

Following the pattern for monads we introduce the corresponding definitions for comonads.
Again $\A$ denotes any category. The following results and definitions are dual to those
in the preceding section.

\begin{thm}{\bf Definition.}\label{def-coBD} \em
For a comonad $\bG=(G,\delta,\ve)$ on $\A$, a comonad distributive law $\kappa:GG\to GG$
(see \ref{co-DL}) satisfying the Yang-Baxter equation is called a {\em comonad BD-law}
or just a {\em BD-law} if the context is clear.
\end{thm}

\begin{proposition}\label{comonads}
Let $\bG=(G, \delta, \ve)$ be a comonad on $\A$ with BD-law $\kappa : GG \to G G$.

\begin{zlist}
\item   ${\bG}^\kappa=(G^\kappa,\delta^\kappa , \ve^\kappa)$ is a comonad on $\A$,
  where $G^\kappa=G$, $\delta^\kappa = \kappa\cdot \delta$ and $\ve^\kappa=\ve$.

\item $\kappa$ defines a comonad distributive law
$\kappa : \bG \bG^\kappa \to \bG^\kappa \bG$ making
the triple $ \bG \bG^\kappa=(G G,{\underline\delta}, {\underline\ve})$ a
comonad  with
$$ \underline \delta : G G\xra{\delta \delta^\kappa}
 GG GG \xra{G \kappa G}  G G G G, \quad
\underline\ve:G G  \stackrel{ \ve \ve}\lra 1 .$$
\item The composite $G \xra{\delta} GG \xra{G\delta}GGG \xra{G \kappa}GGG$
defines a left $ \bG \bG^\kappa$-comodule structure on the functor $G:\A \to \A$.
\item There is a comparison functor \;
 $\overline{K}_\kappa: \A \to \A^{\bG\bG^\kappa}$ \; given by
$$  A\;\mapsto \;(G(A),\,
G(A)\xra{ \,\delta_A\,} GG(A)\xra{G\delta_A} GGG(A)\xra{G\kappa_A}
 GG G(A).$$
\end{zlist}
\end{proposition}

Comonad BD-laws are obtained from monad BD-laws by adjunctions (see \cite[7.4]{MW}):

\begin{thm}\label{mon-comon}{\bf Proposition.}
Let $\bF=(F, m, e)$ be a monad on $\A$ and $\la : F F \to F F$ a BD-law.
If $F$ has a right adjoint $R$, then there is a comonad $\bR=(R,\delta,\ve)$
(where $m\dashv \delta$, $\ve\dashv e$) with
a comonad YB-distributive law $\kappa: R R \to R R$. Moreover, $\la$ is invertible
if and only if $\kappa$ is so.
\end{thm}

\begin{definition}\label{def-Azu-co} \em A comonad $\bG=(G,\delta, \ve)$ on a
category $\A$ is said to be {\em Azumaya} provided it allows for
a (comonad) BD-law $\kappa : G G \to G G$
such that the comparison functor
$\overline{K}_\kappa: \A \to \A^{\bG\bG^\kappa}$ is an equivalence.
\end{definition}

\begin{proposition}\label{l-r-adj-co}
 If $\bG$ is an Azumaya comonad on $\A$, then the functor $G$ admits a right adjoint.
\end{proposition}

This leads to a first characterisation of Azumaya comonads.

\begin{theorem} \label{th.1.co-Monad} Consider a comonad $\bG=(G,\delta,\ve)$ on $\A$
with a comonad BD-law $\kappa : GG \to GG$. The following are equivalent:
\begin{blist}
\item $\bG$ is an Azumaya comonad;

\item the functor $G:\A \to \A$ is comonadic and the left $\bG\bG^\kappa$-comodule structure on
$G$ defined in Proposition \ref{comonads} is Galois;

\item the functor $G:\A \to \A$ is comonadic (for some adjunction $G\dashv R$ with
 counit $\sigma:GR\to 1$) and the composite
$$ GR \xra{\delta R} GGR  \xra{\delta GR} GGG R
\xra{G\kappa R} GGGR \xra{GG\sigma} G G $$
is an isomorphism of comonads  $\mathcal{H} \to \bG\bG^\kappa$, where
$\mathcal{H}$ is the comonad on $\A$ generated by the adjunction $G\dashv R$.
\end{blist}
\end{theorem}

\begin{thm}\label{iso-R}
{\bf The isomorphism
$\A^{\bG\bG^\kappa} \simeq (\A^{\bG^\kappa})^{\wt{\bG}}$.}
\em
Write $\wt{\bG}$ for the lifting of the comonad $\bG$
to $\A^{\bG^\kappa}$ corresponding to the distributive law
$\kappa: \bG\bG^\kappa  \to \bG^\kappa \bG$. Then (see \ref{co-DL}), the assignment
$$ (A, \,  A \xra{\rho} G G(A) ) \; \mapsto \;
((A, A\xra{\,\rho\,} G G(A) \xra{\ve_{G(A)}} G(A) ) ,\,
     A\xra{\,\rho\,} G G(A) \xra{ G (\ve_A)} G(A) )$$
yields an isomorphism of categories
$$\mathcal{Q}_\kappa:\A^{\bG\bG^\kappa} \lra (\A^{\bG^\kappa})^{\wt{\bG}},$$
where for any $((A,\theta),\vartheta) \in (\A^{\bG^\kappa})^{\wt{\bG}}$,
$$\mathcal{Q}_\kappa^{-1}((A,\theta),
\vartheta)=(A, A \xra{\;\vartheta\;} G (A) \xra{G(\theta) } G G (A)).$$

 There is a comparison functor
$ K_\kappa:\A \to (\A^{\bG^\kappa})^{\wt{\bG}}$,
$$ A\;\mapsto \;((G(A),\, G(A)\xra{\delta_A} G G(A)),G(A) \xra{\delta_A}
 G G (A) \xra{\kappa_A} GG(A)),
$$
  with  $\overline{K}_\kappa = \bQ_\kappa^{-1} K_\kappa$ and
commutative diagram
$$
\xymatrix{
\A \ar[r]^-{K_\kappa} \ar[dr]_{\phi^{\bG^\kappa }} & (\A^{\bG^\kappa})^{\wt{\bG}}
\ar[d]^{U^{\wt{\bG}}} \ar[r]^{\bQ_\kappa^{-1}} &
   \A^{\bG\bG^\kappa} \ar[d]^{U^{\bG\bG^\kappa}}  \\
& \A^{\bG^\kappa} \ar[r]_{U^{\bG^\kappa}} & \A.  }
$$
\end{thm}

\begin{thm}\label{iso-R1}{\bf Proposition.}
Let $\bG=(G,\delta,\ve)$ be a comonad on $\A$ and $\kappa : G G \to G G$
an invertible BD-law.
\begin{zlist}
\item $\kappa^{-1}:\bG^\kappa\bG \to \bG\bG^\kappa$ is a comonad distributive law
and hence induces a comonad
$(\bG^\kappa \bG)_{\kappa^{-1}}=(G^\kappa G,\underline{\underline\delta},
 \underline{\underline\ve})$ where
$$ \underline{ \underline{\delta}} : GG \xra{\delta^\kappa\delta}
GG GG \xra{G\kappa^{-1} G}  G G G G , \quad
\underline{\underline{\ve}}:G G \stackrel{\ve \ve }\lra 1,$$
and $\kappa: \bG\bG^\kappa \to (\bG^\kappa \bG)_{\kappa^{-1}}$ is
a comonad isomorphism.
\item There is an isomorphism of categories
$$\varPhi':(\A^{\bG^\kappa})^{(\wt{\bG})_\kappa}\simeq (\A^{\bG})^{(\wt{\bG}^\kappa)_{\kappa^{-1}}},\quad
((A, \theta), \vartheta) \mapsto ((A, \vartheta), \theta).$$
\item $\kappa^{-1}$ induces a comparison functor
$$K'_\kappa:\A \to (\A^{\bG})^{(\wt{\bG^\kappa})_{\kappa^{-1}}}, \;
A\;\mapsto \;((G(A),\, G(A)\xra{\delta_A} GG (A)),
G(A)\xra{\delta^\kappa_A} GG(A)),$$
with commutative diagrams (with $K_\kappa$ from \ref{iso-R})
$$
\xymatrix{ \A \ar[r]^-{K'_\kappa} \ar[dr]_{\phi^{\bG}} &
(\A^{\bG})^{(\wt{\bG^\kappa})_{\kappa^{-1}}} \ar[d]^{U^{(\wt{\bG^\kappa})_{\kappa^{-1}}}}\\
& \A^{\bG}, } \qquad
\xymatrix{ \A \ar[r]^-{K_\kappa} \ar[dr]_{K'_\kappa} & (\A^{\bG^\kappa})^{(\wt{\bG})_\kappa}
\ar[d]^{\varPhi'}\\
&  (\A^{\bG})^{(\wt{\bG^\kappa})_{\kappa^{-1}}} . }
$$
\end{zlist}
\end{thm}

Note that, for $\kappa$ invertible, it follows from the diagrams in the Sections
 \ref{iso-R}, \ref{iso-R1}
that $G$ is an Azumaya comonad if and only if the functor
 $$K'_\kappa:\A \to (\A^{\bG})^{(\wt{\bG^\kappa})_{\kappa^{-1}}}$$
  is an equivalence of categories.
Dualising Proposition \ref{sq} gives:

\begin{proposition}\label{sqd} Let $\bG=(G,\delta,\ve)$ be a comonad on $\A$ with
an invertible  BD-law $\kappa:GG \to GG$ and assume $\kappa^2=1$.
Then the comonad $\bG$ is Azumaya if and only if the comonad $\bG^\kappa$ is so.
\end{proposition}

\begin{thm}\label{Azu-le} {\bf Azumaya comonads with left adjoints.} \em
Again let $\bG=(G,\delta,\ve)$ be a comonad on $\A$ with an invertible  BD-law
 $\kappa:GG \to GG$.
Assume now that the functor $G$ admits a left adjoint functor $L$,
 with $\overline{\eta}, \overline{\varepsilon}: L\dashv G $,
inducing a monad $\mathcal{L}=(L,m,e)$ on $\A$ (see \ref{mod-comod}).
Since $\kappa$ is invertible, $\kappa^{-1}$ can be seen as a distributive
law $\bG^\kappa \bG \to \bG \bG^\kappa$. It then follows from the dual of \ref{DisAd} that the composite
$$\omega:  LG \xra{LG \overline{\eta}} L G G L \xra{L\kappa^{-1} L} LG G  L \xra{\overline{\ve}\,G  L} G  L $$
is a mixed distributive law from the monad $\mathcal{L}$ to the comonad $\bG^\kappa$
leading to an isomorphism of categories
$$ (\A^{\bG^\kappa})^{\wt{\bG}_\kappa}\simeq (\A^{\bG^\kappa})_{\wt{\mathcal{L}}},\quad
 ((A, \theta), \vartheta)\; \mapsto\; ((A, \vartheta),
L(A) \xra{L(\theta)} LG(A)\xra{\overline{\ve}_A} A),$$
where
$\wt{\mathcal{L}}$ is the lifting of $\mathcal{L}$ to $\A^{\bG^\kappa}$
(corresponding to $\omega$).
Then the composite
$$\underline{K_\kappa}:\A \xra{K_\kappa} (\A^{\bG^\kappa})^{\wt{\bG}_\kappa} \xra{\simeq} (\A^{\bG^\kappa})_{\wt{\mathcal{L}}}$$
takes an arbitrary $A \in \A$ to
$$((G(A), \delta^\kappa_A),\, LG(A) \xra{L(\delta_A)} LGG(A)\xra{\overline{\ve}_{G(A)}} A),$$
thus inducing commutativity of the diagram
$$\xymatrix{ \A \ar[rr]^-{\underline{K_\kappa}} \ar[dr]_{\phi^{\bG^\kappa}} & & (\A^{\bG^\kappa})_{\wt{\mathcal{L}}}
\ar[ld]^{U_{\wt{\mathcal{L}}}}\\
& \A^{\bG^\kappa} & . } $$

\begin{theorem} \label{th.3} Let $\bG=(G,\delta,\ve)$ be a comonad on $\A$ with an invertible
comonad BD-law  $\kappa:GG \to GG$ and $\mathcal{L}$ a monad left adjoint to $\bG$
(with $\overline{\eta}, \overline{\varepsilon}: L\dashv G $).
Then the following are equivalent:
\begin{blist}
\item $\bG$ is an Azumaya comonad;
\item \begin{rlist}
\item
the functor $\phi^{\bG^\kappa}:\A \to \A^{\bG^\kappa}$ is monadic  and

\item $\phi^{\bG^\kappa}$ is $\wt{\mathcal{L}}$-Galois, that is,

$t_{(A,\theta)}:L(A) \xra{L(\theta)} LG(A) \xra{L(\delta_A)}LGG(A)
\xra{\overline{\ve}_{G(A)}} G(A), $ is an isomorphism
for any $(A,\theta) \in \A^{\bG^\kappa}$ or

$\chi:LG \xra{L\delta} LGG \xra{L\kappa} LGG
\xra{L\delta G} LGGG\xra{\overline{\ve}GG(A)} GG$ is an isomorphism.
\end{rlist}
\end{blist}
\end{theorem}
\begin{proof} This follows by applying the dual of Theorem \ref{th.1}
to the last diagram.
\end{proof}
\end{thm}

\begin{proposition} \label{cocentral} If $\A$ has coequalisers of reflexive pairs, then
$\underline{K_\kappa}:\A \to (\A^{\bG^\kappa})_{\wt{\mathcal{L}}}$ admits a left adjoint
functor $\underline{L}:(\A^{\bG^\kappa})_{\wt{\mathcal{L}}}\to \A$
whose value at $((A,\vartheta),h)\in (\A^{\bG^\kappa})_{\wt{\mathcal{L}}}$
is given as the coequaliser
$$\xymatrix{
    L(A) \ar@/^1pc/@{->}[rr]^{h} \ar[dr]_{L(\vartheta)}& & A\ar[rr]^-{ q_{((A,\vartheta), h)}} &&\underline{L}((A,\vartheta),h)\\
          &  LG (A) \ar[ru]_{\overline{\ve}_A}&&& .  } $$
\end{proposition}

\begin{thm}{\bf Definition.} \em
Write $G^G$ for the quotient functor of the functor $G$ determined by the coequaliser
of the diagram
$$\xymatrix{
    LG \ar@/^1pc/@{->}[rrrr]^{\overline{\varepsilon}G \cdot L\delta} \ar[drr]_{L\delta}&&&&G\\
          &&  LGG \ar[rru]_{\overline{\varepsilon}G \cdot L\kappa} && . } $$
We call the comonad $\bG$ \emph{cocentral} if $G^G$ is (isomorphic to) the identity functor.
\end{thm}

\begin{theorem}\label{th.2.dual} Assume $\A$ to admit coequalisers of reflexive pairs.
Let $\bG=(G, \delta, \ve)$ be a comonad on  $\A$, $\kappa: GG \to GG$ an
invertible comonad BD-law,
and $\bL$ a monad left adjoint to $\bG$.
Then the comparison functor $\underline{K_\kappa}:\A \to (\A^{\bG^\kappa})_{\wt{\mathcal{L}}}$ is
\begin{rlist}
 \item full and faithful if and only if the comonad $\bG$ is cocentral;

\item an equivalence of categories
      if and only if the comonad $\bG$ is cocentral and the functor $\underline{L}$ is conservative.
\end{rlist}
\end{theorem}

The next observation shows the transfer of the Galois property to an adjoint functor.

\begin{proposition}\label{Galois-Galois} Assume $\bF=(F, m, e)$ to be a monad on $\A$
with invertible BD-law $\la :FF \to FF$, and $\ol\eta,\ol\ve: F\dashv R$ an adjunction
inducing a comonad $\bR=(R,\delta,\ve)$ with invertible BD-law
$\kappa: R R\to R R$ (see Proposition \ref{mon-comon}).
Then the functor $\phi_{\bF^\la}$ is $\widehat{\bR}$-Galois if and only if
the functor $\phi^{\bR^\kappa}$ is $\wt{\mathcal{F}}$-Galois.
\end{proposition}
\begin{proof}
By Theorems \ref{th.1} and \ref{th.3}, we have to show that,
for any $(A,h) \in \A_{\bF^\la}$, the composite
$$t_{(A,h)}:F(A) \xra{\overline{\eta}_{F(A)}}\dF  FF(A) \xra{\dF  (m_A)}\dF  F(A)
\xra{\dF  (h)}\dF  (A)$$ is an isomorphism if and only if,
 for any $(A,\theta) \in \A^{\bR^\kappa}$,
this is so for the composite
$$t_{(A,\theta)}:F(A) \xra{F(\theta)} FR(A) \xra{F(\delta_A)} FRR(A)
\xra{\overline{\ve}_{R(A)}} R(A).$$
By symmetry, it suffices to prove one implication. So suppose that the functor
$\phi_{\bF^\la}$ is $\wt{\bR}$-Galois.
Since $m \dashv \delta$, $\delta$ is the composite
$$R \xra{\overline{\eta}R} RFR \xra{R\overline{\eta}FR } RRFFR
\xra{RRmR} RRFR \xra{RR\overline{\ve}} RR.$$
Considering the diagram
$$
\xymatrix{FR(A) \ar@{=}[dr]\ar[r]^-{F\overline{\eta}_{R(A)}} &
FRFR(A)  \ar[rr]^-{FR\overline{\eta}_{FR(A)} } \ar[d]^{\overline{\ve}_{FR(A)}}&&
FR^2F^2R(A)  \ar[rr]^-{FR^2m_{R(A)}} \ar[d]^{\overline{\ve}_{RF^2R(A)}}&&
FR^2FR(A) \ar[r]^-{FR^2\overline{\ve}_A} \ar[d]^{\overline{\ve}_{RFR(A)}} &
FR^2(A) \ar[d]^{\overline{\ve}_{R(A)}}\\
F(A) \ar[dr]_{\overline{\eta}_{F(A)}}\ar[u]^{F(\theta)}&
FR(A) \ar[rr]^{\overline{\eta}_{FR(A)}}&&RF^2R(A) \ar[rr]^{Rm_{R(A)}}&&
RFR(A) \ar[r]^{R\overline{\ve}_A}& R(A)\\
& RF^2(A) \ar[rru]_{RF^2(\theta)} \ar[rr]_{Rm_A}&& RF(A) \ar[rru]_{RF(\theta)}}
$$
in which the top left triangle commutes by one of the triangular identities
for $F \dashv R$ and the other partial diagrams commute by naturality, one sees that
$t_{(A,\theta)}$
is the composite
$$F(A) \xra{\overline{\eta}_{F(A)}} RFF(A) \xra{Rm_A} RF(A)
\xra{RF(\theta)}RFR(A) \xra{R\,\overline{\ve}_A} R(A).$$
Since $(A,\theta) \in \A^{\bR^\kappa}$, the pair $(A, F(A) \xra{F(\theta)}FR(A) \xra{\overline{\ve}_A} A)$ --
being $\Psi^{-1}(A, \theta)$ (see \ref{mod-comod}) --
is an object of the category
$\A_{\bF^\la}$. It then follows that $t_{(A,\theta)}=t_{(A,\overline{\ve}_A \cdot F(\theta))}$.
Since the functor $\phi_{\bF^\la}$ is assumed to be $\wt{\bR}$-Galois,
the morphism $t_{(A,\overline{\ve}_A \cdot F(\theta))}$,
and hence also $t_{(A,\theta)}$, is an isomorphism, as desired.
\end{proof}

In view of the properties of separable functors (see \ref{adj-sep}) and
Definition \ref{def-azu},
for an Azumaya monad $\bF$, $\bF \bF^\la$ is a separable monad
if and only if $F$ is a separable functor.
In this case $\phi_{\bF^\la}$ is also a separable functor, that is, the unit $e:1\to F$
splits.

Dually, for an Azumaya comonad $\bR$, $\bR \bR^\kappa$ is separable
if and only if the functor $R$ is separable. Thus we have:

\begin{theorem}\label{Azu-mon-comon} Under the conditions of Proposition \ref{Galois-Galois},
suppose further that  $\A$ is a Cauchy complete category. Then the following are equivalent:
\begin{blist}
\item $(\bF,\la)$ is an Azumaya monad and $\bF\bF^\la$ is a separable monad;
\item $(\bF,\la)$ is an Azumaya monad and the unit $e:1\to F$ is a split monomorphism;
\item $\phi_{\bF^\la}$ is $\widehat{\bR}$-Galois and $e:1\to F$ is a split monomorphism;
\item $(\bR,\kappa)$ is an Azumaya comonad and the counit $\ve:R\to 1$ is a split epimorphism;
\item $\phi^{\bR^\kappa}$ is $\wt{\mathcal{F}}$-Galois and $\ve:R\to 1$ is a split epimorphism.
\item $\phi^{\bR^\kappa}$ is $\wt{\mathcal{F}}$-Galois and $\bR\bR^\kappa$ is a separable
   comonad.
\end{blist}
\end{theorem}
\begin{proof}
(a)$\Ra$(b)$\Ra$(c) follow by the preceding remarks.

(c)$\Ra$(a) Since $\A$ is assumed to be Cauchy complete, by \cite[Corollary 3.17]{Me},
the splitting of $e$ implies that the functor $\phi_{\bF^\la}$ is comonadic.
Now the assertion follows by Theorem \ref{th.1}.

Since $\ve$ is the mate of $e$,
$\ve$ is a split epimorphism if and only if $e$ is a split monomorphism
(e.g. \cite[7.4]{MW})
and the splitting of $\ve$ implies that the functor $\phi^{\bR^\kappa}$ is monadic.

Applying now Theorems \ref{th.1}, \ref{th.3} and Proposition \ref{Galois-Galois} gives the desired
result.
\end{proof}

\section{Azumaya algebras in  braided monoidal categories}\label{azu-braid}

\begin{thm}{\bf Algebras and modules in monoidal categories.} \em
Let $(\V, \ot, I, \tau)$ be a strict monoidal category (\cite{Mc}).
An \emph{algebra} $\mathcal{A}=(A, m, e)$ in  $\V$ (or $\V$-\emph{algebra})
 consists of an object $A$ of $\V$ endowed with  multiplication $m:
A \otimes A \to A$  and unit morphism $e : I \to A$ subject to the usual identity
and associative conditions.

For a $\V$-algebra $\mathcal{A}$, a \emph{left} $\mathcal{A}$-\emph{module} is a
pair $(V, \rho_V)$, where $V$ is
an object of $\V$ and $\rho_V: A \ot V \to V$ is a morphism in $\V$,
called the \emph{left action}
(or $\mathcal{A}$-\emph{left action}) on $V$, such that
$\rho_V(m \otimes V)=\rho_V(A \ot \rho_V)$ and $\rho_V (e \ot V)=1$.

Left $\mathcal{A}$-modules are objects of a category ${_{\mathcal{A}}\V}$ whose
 morphisms between objects
$f:(V,\rho_V)\to (W,\rho_W)$ are morphism $f: V \to W$ in $\V$
such that $\rho_W (A \ot f)=f\rho_V$.
Similarly, one has the category $\V_{\mathcal{A}}$ of right $\mathcal{A}$-modules.

The forgetful functor ${_{\mathcal A}U} \colon {_{\mathcal A}\V} \to \V $,
  taking a left $\mathcal{A}$-module
$(V,\rho_V)$ to the object $V$, has a left adjoint, the {\em free ${\mathcal A}$-module functor}
$$\phi_A:\V \to _{\mathcal A}\!\!\V,\quad V \;\mapsto \;
(A \ot V, m_A \ot V).$$

There is another way of representing the category of left $\mathcal{A}$-modules
involving modules over the monad associated to the $\V$-algebra $\mathcal{A}$.

Any $\V$-algebra $\mathcal{A}=(A, m, e)$
defines a monad $\mathcal{A}_l=(T, \eta, \mu)$ on $\V$ by putting

\begin{itemize}
\item $\emph{T}(V)=A\otimes V$,
\item$\eta_V =e\otimes V : V \to A \otimes V$,
\item $\mu_V=m \otimes V : A\otimes A \otimes
V \to A\otimes V $.
\end{itemize}

The corresponding Eilenberg-Moore category
$\V_{\mathcal{A}_l}$ of $\mathcal{A}_l$-modules is exactly the category
${_\mathcal{A}\V }$ of left $\mathcal{A}$-modules, and
${_{\mathcal A}U} \dashv F$ is the familiar forgetful-free adjunction between
$\V_{\mathcal{A}_l}$ and $\V$. This gives in particular that the forgetful functor
${_{\mathcal A}U} \colon _{\mathcal A}\V \to \V $ is
monadic. Hence the functor $_{\mathcal A}U$ creates those limits that exist in $\V$.

Symmetrically, writing $\mathcal{A}_r$ for the monad on $\V$ whose functor part is
$-\ot A$,
the category $\V_\mathcal{A} $ is isomorphic to the  Eilenberg-Moore category
$\V_{\mathcal{A}_r}$ of $\mathcal{A}_r$-modules,
and the forgetful functor $U_{\mathcal A} \colon \V_{\mathcal A} \to \V $
is monadic and creates those limits that exist in $\V$.

If $\V$ admits coequalisers, $\mathcal{A}$ is a $\V$-algebra, $(V,\varrho_V) \in \V_{\mathcal{A}}$ a right
$\mathcal{A}$-module, and $(W,\rho_W) \in {_{\mathcal{A}}\V}$ a left $\mathcal{A}$-module, then their
\emph{tensor product} (\emph{over} $\mathcal{A}$) is the object part of the coequaliser
$${\xymatrix{ V \ot A \ot W \ar@{->}@<0.5ex>[rr]^-{\varrho_V \ot W} \ar@
{->}@<-0.5ex> [rr]_-{V \ot \rho_W}&& V \ot W \ar[r]& V \ot_A W.}}$$
\end{thm}

\begin{thm}{\bf Bimodules.} \em
If $\bA$ and $\bB$ are $\V$-algebras, an object $V$ in $\V$ is called an
$(\mathcal{A},\mathcal{B})$-\emph{bimodule} if there are morphisms
$\rho_V: A \ot V \to V$ and $\varrho_V: V \ot B \to V$ in $\V$ such that
$(V, \rho_V) \in {_{\mathcal{A}}\V}$, $(V, \varrho_V) \in \V_{\mathcal{B}}$ and
$\varrho_V(\rho_V\ot B)=\rho_V (A \ot \varrho_V)$.
A morphism of $(\mathcal{A},\mathcal{B})$-bimodules is a morphism in $\V$ which is a
morphism of  left $\mathcal{A}$-modules as well as of right
$\mathcal{B}$-modules. Write ${_{\mathcal{A}}\V}_{\mathcal{B}}$ for
the corresponding category.

Let $\mathcal{I}$ be the trivial $\V$-algebra $(I, 1_I: I=I
\ot I \to I, 1_I : I \to I)$. Then, $_{\mathcal{I}}\V=\V_{\mathcal{I}}=\V$, and for any
$\V$-algebra
$\mathcal{A}$, the category ${_{\mathcal{A}}\V}_{\mathcal{I}}$ is (isomorphic to)
the category of left $\mathcal{A}$-modules
$_A\V$, while the category ${_{\mathcal{I}}\V}_{\mathcal{A}}$ is (isomorphic to)
the category of right $\mathcal{A}$-modules
$\V_A$. In particular, ${_{\mathcal{I}}\V}_{\mathcal{I}}=\V$.
\end{thm}

\begin{thm}{\bf The monoidal category of bimodules.} \em Suppose now that $\V$
admits coequalisers and that the tensor product preserves these coequaliser in
both variables (i.e. all functors $V \ot -:\V \to \V$
as well as $-\ot V:\V \to \V$ for $V\in \V$ preserve coequalisers).
The last condition guarantees
that if $\mathcal{A}$, $\mathcal{B}$ and $\mathcal{C}$ are $\V$-algebras and if
$M \in {_{\mathcal{A}}\V}_{\mathcal{B}}$ and
$N \in {_{\mathcal{B}}\V}_{\mathcal{C}}$, then
\begin{itemize}
  \item $M \ot_B N \in {_{\mathcal{A}}\V}_{\mathcal{C}}$;
  \item if $\mathcal{D}$ is another $\V$-algebra and
 $P \in {_{\mathcal{C}}\V}_{\mathcal{D}}$, then the
  canonical morphism
$$(M \ot_B N) \ot_C P \to M \ot_B (N \ot_C P)$$ induced by the associativity of the tensor product,
  is an isomorphism in ${_{\mathcal{A}}\V}_{\mathcal{D}}$;
  \item $({_{\mathcal{A}}\V}_{\mathcal{A}}, -\ot_A-, \mathcal{A})$ is a monoidal category.
\end{itemize}
Note that (co)algebras in this monoidal category are called
 $\mathcal{A}$--\emph{(co)rings}.
\end{thm}

\begin{thm}\label{Coalg.comod}{\bf Coalgebras and comodules in monoidal categories.} \em
Associated to
any monoidal category $\V=(\V,\ot, I)$, there are three monoidal categories
$\V^{\rm{op}}$, $\V^r$ and $(\V^{\text{op}})^r$ obtained
from $\V$ by reversing, respectively, the morphisms, the tensor product and both
the morphisms and tensor
product, i.e., $\V^{\text{op}}=(\V^{\text{op}}, \ot, I)$,
$\V^r=(\V, \ot^r, I)$, where $V \ot^r W:=W \ot V$, and
$(\V^{\text{op}})^r=(\V^{\text{op}}, \ot^r, I)$ (see, for example, \cite{Saa}).
Note that $(\V^{\text{op}})^r=(\V^r)^{\text{op}}$.

\emph{Coalgebras} and \emph{comodules} in a monoidal category $\V=(\V,\ot, I)$
are respectively algebras and modules in  $\V^{\text{op}}=(\V^{\text{op}},\ot, I)$.
If $\mathcal{C}=(C, \delta, \varepsilon)$ is a $\V$-coalgebra, we write
$\V^\mathcal{C}$ (resp. $^\mathcal{C}\V$) for the category of right (resp. left)
$\mathcal{C}$-comodules.
Thus, $\V^\mathcal{C}=(\V^{\text{op}})_\mathcal{C}$ and
$^\mathcal{C}\V={_{\mathcal{C}}(\V^{\text{op}})}$. Moreover, if $\mathcal{C}'$ is another
$\V$-coalgebra, then the category $^\mathcal{C}\V^{\mathcal{C}'}$ of
$(\mathcal{C}, \mathcal{C}')$-bicomodules
is $_\mathcal{C} (\V^{\text{op}}) _{\mathcal{C}'}$. Writing $\mathcal{C}_l$ (resp. $\mathcal{C}_r$)
for the comonad on $\V$ with functor-part $C \ot -$ (resp. $-\ot C$), one has that
$\V^\mathcal{C}$ (resp. $^\mathcal{C}\V$) is just the category of $\mathcal{C}_l$-comodules (resp. $\mathcal{C}_r$-comodules).
\end{thm}

\begin{thm}\label{duality}{\bf Duality in monoidal categories.} \em
One says that an object $V$ of $\V$ \emph{admits a left dual}, or {\em left adjoint},
 if there exist
an object $V^*$ and morphisms $\db : I \to V \ot V^*$ and $\ev :V^* \ot V \to I$
such that the composites
$$V \xra{ \db \ot V} V \ot V^* \ot V \xra{V \ot \ev  } V , \quad
 V^* \xra{V^* \ot \db} V^* \ot V \ot V^* \xra{\ev  \ot V^*} V^*,$$ yield
the identity morphisms.  $\db$ is called the \emph{unit} and
$\ev $ the \emph{counit} of the adjunction. We use the notation $(\db, \ev: V^* \dashv V)$
to indicate that $V^*$ is left adjoint to $V$ with unit $\db$ and counit $\ev$.
This terminology is justified by the fact that such an adjunction
induces an adjunction of functors
$$\db\ot \,-\, , \ev \ot \,-\,: V^* \ot \,-\, \dashv
V\ot \,-\,:\V \to \V,$$
as well as an adjunction of functors
$$\,-\,\ot \db , \,-\,\ot \ev : \,-\, \ot V \dashv \,-\, \ot V^* :\V\to \V,$$
i.e., for any $X,Y \in \V$, there are bijections
$$\V(V^* \ot X, Y)\simeq \V(X, V \ot Y)\quad \text{and} \quad \V(X \ot V, Y)\simeq \V(X, Y \ot V^*)
.$$

Any adjunction $(\db, \ev: V^* \dashv V)$, induces a $\V$-algebra and a $\V$-coalgebra,
$$\begin{array}{l}
   \mathcal{S}_{V ,V^* }=(V \ot V^* ,\,   V\ot   V^*  \ot   V \ot   V^*
\xra{V^*\ot \ev \ot V } V \ot   V^* ,\, \db:I \to V \ot   V^* ) , \\[+1mm]
   \mathfrak{C}_{V^*,V }=(V \ot V^*,\, V \ot   V^*
\xra{V^* \ot \db \ot V} V \ot   V^* \ot   V \ot   V^*, \, \ev :V^* \ot V \to I) .
\end{array} $$

Dually, one says that an object $V$ of $\V$ \emph{admits a right adjoint} if there exist
an object $V^\sharp$ and morphisms $\db' : I \to V^\sharp \ot V$ and $\ev' :V \ot V^\sharp \to I$
such that the composites
$$V^\sharp \xra{ \db \ot V^\sharp} V^\sharp \ot V \ot V^\sharp \xra{V^\sharp \ot \ev} V^\sharp , \quad
 V \xra{V \ot \db} V \ot V^\sharp \ot V \xra{\ev  \ot V} V, $$ yield
the identity morphisms.

\end{thm}
\begin{proposition}\label{galois.m}
Let $V\in \V$ be an object with a left dual $(V^*, \db, \ev)$.
\begin{rlist}
\item For a $\V$-algebra $\mathcal{A}$ and a left $\mathcal{A}$-module structure $\rho_V:A \ot V \to V$ on $V$,
 the morphism
$$t_{(V,\rho_V)}: A \xra{A \ot {\rm db}} A \ot V \ot V^* \xra{\rho_V \ot V^*} V \ot V^*$$
(mate of $\rho_V$ under $\V(A \ot V, V)\simeq \V(A, V \ot V^*)$)
is a morphism from the $\V$-algebra $\mathcal{A}$
to the $\V$-algebra $\mathcal{S}_{V,V^*}$.

\item For a $\V$-coalgebra $\mathcal{C}$ and a right $\mathcal{C}$-comodule
structure $\varrho_V:V \to V \ot C$, the morphism
$$t^c_{(V,\varrho_V)}: V^* \ot V \xra{V^* \ot \varrho_V} V^* \ot V \ot C
\xra{{\rm ev}\ot C} C$$
(mate of $\varrho_V$ under $ \V(V, V \ot C) \simeq \V(V^* \ot V, C)$)
is a morphism from the $\V$-coalgebra $\mathfrak{C}_{V,V^*}$
to the $\V$-coalgebra $\mathcal{C}$.
\end{rlist}
\end{proposition}

\begin{thm}\label{gal-def}{\bf Definition.} \em Let $V\in \V$ be an object with a left dual $(V^*, \db, \ev)$.
\begin{rlist}
  \item  For a $\V$-algebra $\mathcal{A}$, a left $\mathcal{A}$-module
$(V,\rho_V)$ is called \emph{Galois} if  
the morphism
$t_{(V,\rho_V)}:A \to V \ot V^*$  is an isomorphism
between the $\V$-algebras $\mathcal{A}$ and $\mathcal{S}_{V,V^*}$,
and  \emph{faithfully Galois}
if, in addition, the functor $V \ot -:\V \to \V$ is conservative.
 \item  For a $\V$-coalgebra $\mathcal{C}$, a right $\mathcal{C}$-comodule
$(V,\varrho_V)$ is called \emph{Galois} if  
the morphism
$t^c_{(V,\varrho_V)}: V^* \ot V \to C$  is an isomorphism between the $\V$-coalgebras
 $\mathfrak{C}_{V,V^*}$ and $\mathcal{C}$,
and  \emph{faithfully Galois}
if, in addition, the functor $V \ot -:\V \to \V$ is conservative.
\end{rlist}
\end{thm}

\begin{thm}\label{braid-mon}{\bf Braided monoidal categories.} \em A braided monoidal category
is a quadruple $(\V, \ot, I, \tau)$, where
$(\V, \ot, I)$ is a monoidal category and $\tau$, called the \emph{braiding}, is a
collection of natural isomorphisms
$$\tau_{V,W}: V\ot W \to W\ot V, \quad V,W\in \V,$$
subject to two hexagon coherence identities (e.g. \cite{Mc}).
A \emph{symmetric} monoidal category
is a monoidal category with a braiding $\tau$ such that
$\tau_{V,W} \cdot \tau_{W,V}=1$ for all $V,W \in \V$.
If $\V$ is a braided category with braiding $\tau$, then the monoidal category
$\V^r$ becomes a braided category with braiding given by
$\overline{\tau}_{V, W}:=\tau_{W,V}$.
Furthermore, given $\V$-algebras $\mathcal{A}$ and $\mathcal{B}$, the triple
$$\mathcal{A} \ot \mathcal{B}=(A\ot B, (m_A\ot m_B)\cdot (A\ot\tau_{B,A}\ot B), e_A\ot e_B)$$
is also a $\V$-algebra, called the \emph{braided tensor product of $\mathcal{A}$ and
$\mathcal{B}$}.

The braiding also has the following consequence (e.g \cite{St}):
\smallskip
\begin{itemize}
\item[]
{\em If an object $V$ in $\V$ admits a left dual
$(V^*, \db : I \to V \ot V^*,\ev :V^* \ot V \to I),$
then $(V^*,\db',\ev')$ is right adjoint
to $V$ with unit and counit
$$\db':I \xra{\db } V\ot V^* \xra{\tau^{-1}_{V^*,V}} V^* \ot V,\quad
\ev':V \ot V^* \xra{\tau_{V,V^*}}V^* \ot V \xra{{\rm ev}}I.$$ }
\end{itemize}
\end{thm} Thus there are isomorphisms $(V^*)^\sharp \simeq V$ and $(V^\sharp)^* \simeq V$, and we have the following

\begin{thm}\label{finite-object}{\bf Definition.} \em An object $V$ of a braided monoidal
category $\V$ is said to be \emph{finite} if $V$ admits a left
(and hence also a right) dual.
\end{thm}

For the rest of this section, $\V=(\V, \ot, I, \tau)$ will denote  a braided
monoidal category.

\medskip

Finite objects in a braided monoidal category have the following relationship between
the related functors to be (co)monadic or conservative. Recall that a morphism $f: V \to W$ in $\V$
called a \emph{copure epimorphism} (\emph{monomorphism}) if for any $X \in \V$, the morphism $f\ot X:V \ot X \to W\ot X$
(and hence also the morphism $X\ot f:X \ot V \to X\ot W$) is a regular epimorphism (monomorphism).

\begin{proposition} \label{finite.reg} Let $\V$ be a braided monoidal category admitting
equalisers and coequalisers. For a finite object
$V \in \V$ with left dual $(V^*,  {\rm db}, {\rm ev})$,
the following are equivalent:
\begin{blist}
  \item  $V \ot -:\V \to \V$ is conservative (monadic, comonadic);
  \item ${\rm ev}:V^* \ot V \to I$ is a copure epimorphism,
   \item  $- \ot V:\V \to \V$ is conservative (monadic, comonadic);
  \item ${\rm db}:I \to V \ot V^*$ is a pure monomorphism.
\end{blist}
\end{proposition}
\begin{proof}  First observe that since $V$ is assumed to admit a left dual, it
admits also a right dual (see \ref{braid-mon}).
Hence the equivalence of the properties listed in (a) (and in (c)) follows from
\ref{lem-mon}. It only remains to show the equivalence of (a) and (b), since the equivalence
of (c) and (d) will then follow by duality.

(a)$\Ra$(b) If $V \ot -:\V \to \V$ is monadic, then it follows from \cite[Theorem 2.4]{KP}
that each component of the counit of the adjunction  $V^*\ot - \dashv V \ot -$,
which is the natural transformation $\ev  \ot -$, is a regular epimorphism. Thus,
$ {\rm ev}:V^* \ot V \to I$ is a copure epimorphism.

(b)$\Ra$(a) To say that $\ev :V^* \ot V \to I$ is a copure epimorphism is to say that each component of
the counit $\ev  \ot -$ of the adjunction $V^*\ot - \dashv V \ot -$ is a regular epimorphism,
which implies (see, for example, \cite{KP}) that $V \ot -:\V \to \V$ is conservative.
\end{proof}

\begin{remark} \label{finite.reg-rem} \em In Proposition \ref{finite.reg},
if the tensor product preserves regular epimorphisms, then   (b)
is equivalent to require $\ev :V^* \ot V \to I$ to be a regular epimorphism.

If the tensor product in $\V$ preserves regular monomorphisms, then  (d)
is equivalent to require ${\rm db}:I \to V \ot V^*$ to be a regular monomorphism.
\end{remark}

\begin{thm}\label{op-alg}{\bf Opposite algebras.} \em
  For a $\V$-algebra $\mathcal{A}=(A, m, e)$, define the {\em opposite} algebra
$\bAp=(A,m^\tau,e^\tau)$ in $\V$ with multiplication
$m^\tau=m \cdot \tau_{A,A}$ and unit $e^\tau=e$.
Denote by
$\mathcal{A}^e = \mathcal{A} \ot \bAp$ and
by ${^e\!\mathcal{A}}= \bAp \ot \mathcal{A}$ the braided tensor products.

Then ${A}$ is a left $\mathcal{A}^e$-module as well as a right ${^e\!\mathcal{A}}$-module
by the structure maps
$$\begin{array}{l}
A\ot A^\tau \ot A \xrightarrow{A\ot \tau_{A,A}} A\ot A\ot A\xrightarrow{A\ot m}
 A\ot A \xrightarrow{m} A, \\[+1mm]
A\ot A^\tau \ot A \xrightarrow{\tau_{A,A}\ot A} A\ot A\ot A\xrightarrow{m\ot A}
 A\ot A \xrightarrow{m} A.
\end{array}$$

By properties of the braiding, the morphism
$\tau_{A,A}: A   \ot  A \to A \ot A$ induces a distributive law
from the monad $  (\mathcal{A}^\tau)_l $ to the monad $ \mathcal{A}_l $
satisfying the Yang-Baxter equation and the monad
$ \mathcal{A}_l (\mathcal{A}^\tau)_l$
is just the monad ${(\bA^e)_l}$. Thus the category of
$\mathcal{A}_l (\mathcal{A}^\tau)_l$-modules
is the category ${_{\mathcal{A}^e}\V}$ of left $\mathcal{A}^e$-modules.
Symmetrically, the category of $\mathcal{A}_r (\mathcal{A}^\tau)_r$-modules
is the category $\V_{^e\!\mathcal{A}}$ of right $^e\!\mathcal{A}$-modules.
\end{thm}

\begin{thm}\label{azu-def}{\bf Azumaya algebras.} \em
Given a $\V$-algebra $\bA=(A,m,e)$,
by Proposition \ref{prop.1}, there are
two comparison functors
$$\overline{K}_l:\V \to \V_{\mathcal{A}_l (\mathcal{A}^\tau)_l}=
 {_{\mathcal{A}^e}\!\V}, \quad
\overline{K}_r:\V \to \V_{\mathcal{A}_r (\mathcal{A}^\tau)_r}=
 {\V_{^e\!\mathcal{A}}},$$
given by the assignments
$$\overline{K}_l:V \; \longmapsto \; (A\ot V,\, A \ot A\ot A \ot V \xra{A \ot m^\tau\ot V}
 A \ot A \ot V \xra{m \ot V} A  \ot V),$$
$$\overline{K}_r:V \; \longmapsto \; (V\ot A,\, V \ot A\ot A \ot A \xra{V\ot m^\tau\ot A}
V\ot A \ot A \xra{V\ot m} V  \ot A)$$
with commutative diagrams
\begin{equation}\label{th.1.Alg.D}
\xymatrix{ \V \ar[rr]^-{\overline{K}_l} \ar[dr]_{A\ot -}
  & &{_{\mathcal{A}^e}\V} \ar[dl]^{_{\mathcal{A}^e}U} \\
&\V\,& ,}\qquad
\xymatrix{ \V \ar[rr]^-{\overline{K}_r} \ar[dr]_{-\ot A}
 & & \V_{^e\!\mathcal{A}} \ar[dl]^{U_{_{^e\!\mathcal{A}}}} \\
&\V &.}\end{equation}

The $\V$-algebra $\bA$ is called {\em left (right) Azumaya}
provided $\mathcal{A}_l$ ($\mathcal{A}_r$) is an Azumaya monad.
\end{thm}

\begin{remark}\label{sq.alg} \em
It follows from Remark \ref{sq} that if $\tau^2_{A,A}=1$,
 the monad $\bA_l$ (resp. $\bA_r$) is Azumaya if and only if $(\bA^\tau)_l$ (resp. $(\bA^\tau)_l)$ is.
Thus, in a symmetric monoidal category, a $\V$-algebra is left (right) Azumaya if
ond only if its opposite is so.
\end{remark}

A basic property of these algebras is the following.

\begin{proposition} \label{finite} Let $\V$ be a braided monoidal category and
 $\bA=(A,m,e)$ a $\V$-algebra. If $\bA$ is left Azumaya, then  $A$ is finite in $\V$.
\end{proposition}
\begin{proof} It is easy to see that when $\V$ and ${_{\mathcal{A}^e}\!\V}$ are
viewed as right $\V$-categories
(in the sense of \cite{P2}), $\overline{K}_l$ is a $\V$-functor. Hence, when
$\overline{K}_l$ is an equivalence of
categories (that is, when $\bA$ is left Azumaya), its inverse equivalence
$\overline{R}$ is also a $\V$-functor. Moreover,
since $\overline{R}$ is left adjoint to $\overline{K}_l$, it preserves all
colimits that exist in $_{\mathcal{A}^e}\V$.
Obviously, the functor $\phi_{(\mathcal{A}^e)_l}:\V \to {_{\mathcal{A}^e}\V}$
is also a $\V$-functor and, moreover, being
a left adjoint, it preserves all colimits that exist in $\V$.
Consequently, the  composite
$\overline{R} \cdot \phi_{(\mathcal{A}^e)_l}:\V \to \V$ is a $\V$-functor and
preserves all colimits that exist in $\V$.
It then follows from \cite[Theorem 4.2]{P2} that there exists an object $A^*$ such that
$\overline{R} \cdot \phi_{(\mathcal{A}^e)_l}\simeq A^* \ot -$. Using now
that $\overline{R} \cdot \phi_{(\mathcal{A}^e)_l}$
is left adjoint to the functor $A \ot -:\V \to \V$, it is not hard to see that
$A^*$ is a left dual to $A$.
\end{proof}

\begin{thm}\label{th.1.Alg}{\bf Left Azumaya algebras.}
 Let $(\V,\ot, I, \tau)$ be a braided monoidal category and $\bA=(A,m,e)$
a $\V$-algebra. The following are equivalent:
\begin{blist}
\item  $\bA$ is a left Azumaya algebra;
\item the functor $A \ot -:\V \to \V$ is monadic and the left $(\mathcal{A}^e)_l$-module structure
   on it induced by the left diagram in (\ref{th.1.Alg.D}), is Galois.
\item  {\rm (i)} $A$ is finite with left dual
$(A^*, {\db}:I \to A\ot A^*,{\ev}:A^*\ot A \to I)$,
  the functor $A \ot -:\V \to \V$ is monadic and \\[+1mm]
 {\rm (ii)}   the composite
  $\overline{\chi}_0:$
$$ A\ot A \xra{A\ot A\ot \db} A\ot A\ot A \ot A^* \xra{A \ot \tau_{A, A}\ot A^*}
 A \ot A\ot A \ot A^*   \xra{m \ot  A \ot A^*} A\ot A \ot A^* \xra{m \ot A^*} A\ot A^* $$
is an isomorphism (between the $\V$-algebras $\mathcal{A}^e$ and $\mathcal{S}_{A,A^*})$;

\item
 {\rm (i)} $A$ is finite with right dual $(A^\sharp,  {\rm db}': I \to A^\sharp \ot A, {\rm ev}':A \ot A^\sharp \to I)$,
 the functor $\phi_{(\bA^\tau)_l}:\V \to \V_{(\bA^\tau)_l}=_{\bA^\tau}\!\!\V$ is comonadic and \\[+1mm]
 {\rm (ii)}   the composite
  $\overline{\chi}:$
$$ A\ot A \xra{\db'\ot A\ot A} A^\sharp\ot A\ot A \ot A \xra{A^\sharp \ot m\ot A}
 A^\sharp \ot A\ot A     \xra{A^\sharp \ot \tau_{A,A } } A^\sharp \ot A \ot A\xra{ A^\sharp\ot m} A^\sharp\ot A $$
is an isomorphism.
\end{blist}
\end{thm}
\begin{proof}
(a)$\LRa$(b)
 This follows by Proposition \ref{P.1.2}.

(a)$\LRa$(c)
If $\bA$ is a left Azumaya algebra, then $A$ has a left dual by Proposition \ref{finite}.
Thus, in both cases, $A$ is finite, i.e. there is an adjunction $(\db , \ev : A^* \dashv A)$.
Then the functor $A^*\ot -:\V \to \V$ is left adjoint to the functor $A\ot -:\V \to \V$,
and the monad on $\V$ generated by this adjunction is $(\mathcal{S}_{A,A^*})_l$.
It is then easy to see that the monad morphism $t_{\overline{K}_l}:(\mathcal{A}^e)_l \to (\mathcal{S}_{A,A^*})_l$
corresponding to the left commutative diagram in (\ref{th.1.Alg.D}), is just $\overline{\chi}_0 \ot-$.
Thus, $t_{\overline{K}_l}$ is an isomorphism if and only if $\overline{\chi}_0$ is so.
It now follows from Theorem \ref{th.1.Monad} that (a) and (c) are equivalent.

(a)$\LRa$(d)  Any left Azumaya algebra has a left (and hence also a right) dual by Proposition \ref{finite}.
Moreover, if $A$ has a right dual $A^\sharp$, then the functor $A^\sharp \ot -$ is right adjoint to
the functor $A \ot -$. The desired equivalence now follows by applying Theorem \ref{th.1} to the monad $\bA_l$ and using that
the natural transformation $\chi$ is just $\overline{\chi} \ot -$.
\end{proof}

Each statement about a general braided monoidal category $\V$ has a counterpart statement
obtained by interpreting it in $\V^r$. We do this for Theorem \ref{th.1.Alg}.

\begin{thm}\label{th.1.Alg.rev}{\bf Right Azumaya algebras.}
 Let $(\V,\ot, I, \tau)$ be a braided monoidal category and $\bA=(A,m,e)$
a $\V$-algebra. The following are equivalent:
\begin{blist}
 \item $\bA$ is right Azumaya;

 \item the functor $- \ot A:\V \to \V$ is monadic and the right $({^e\!\mathcal{A}})_r$-module structure
   on it induced by the  right diagram in (\ref{th.1.Alg.D}), is Galois;
 \item {\rm(i)} $A$ is finite with right dual $(A^\sharp,  {\rm db}': I \to A^\sharp \ot A, {\rm ev}':A \ot A^\sharp \to I)$,
 the functor $-\ot A : \V \to \V$ is monadic and \\[+1mm]
{\rm(ii)} the composite $\overline{\chi}_1:$
$$ A\ot A \xra{\db'\ot A\ot A} A^\sharp\ot A\ot A \ot A \xra{A^\sharp \ot \tau_{A, A}\ot A}
 A^\sharp \ot A\ot A \ot A   \xra{  A^\sharp \ot m \ot A} A^\sharp\ot A \ot A \xra{A^\sharp \ot m } A^\sharp\ot A$$
is an isomorphism (between the $\V$-algebras ${^e\!\mathcal{A}}$ and $\mathcal{S}_{A^\sharp ,A})$.

  \item {\rm (i)} $A$ is finite with left dual $(A^*,  {\db}: I \to A \ot A^*, {\ev}:A^* \ot A \to I)$,
  the functor $\phi_{(\bA^\tau)_r}:\V \to \V_{(\bA^\tau)_r}=\V_{\bA^\tau}$ is comonadic and \\[+1mm]
 {\rm (ii)}   the composite $\overline{\chi}_2:$ $$ A\ot A \xra{A\ot A\ot \db} A\ot A\ot A \ot A^* \xra{A \ot m \ot A^*}
 A\ot A \ot A^*   \xra{\tau_{A, A} \ot A^*} A\ot A \ot A^* \xra{m \ot A^*} A\ot A^* $$
is an isomorphism.
\end{blist}
\end{thm}

\begin{proposition}\label{l-r-azu}
In any braided monoidal category, an algebra is left (right) Azumaya if and only if
its opposite algebra is right (left) Azumaya.
\end{proposition}
\begin{proof} We just note that if $(\V,\ot, I, \tau)$ is a braided monoidal category and $\bA$ is a $\V$-algebra, then
$(\tau_{-,A})^{-1}: A \ot \,-\, \to \,-\, \ot A^\tau$
is an isomorphism of monads $\bA_l \to (\bA^\tau)_r$, while $(\tau_{A,-})^{-1}: \,-\,\ot A \to A^\tau \ot \,-\, $
is an isomorphism of monads $\bA_r \to (\bA^\tau)_l$.
\end{proof}

Under some conditions on $\V$, left Azumaya algebras are also right Azumaya and vice versa:

\begin{thm}\label{th.1.Alg.1}{\bf Theorem.}
 Let $\bA=(A,m,e)$ be a $\V$-algebra
in a braided monoidal category $(\V,\ot, I, \tau)$ with equalisers and coequalisers.
 Then the following are equivalent:
\begin{blist}
\item  $\bA$ is a left Azumaya algebra;
\item the left $\mathcal{A}^e$-module $(A, m \cdot (A \ot m^\tau))$
    is faithfully Galois;
\item $A$ is finite with right dual $(A^\sharp,  {\rm db}':I\to A^\sharp \ot A,
 {\rm ev}':A \ot A^\sharp \to I)$,
  the functor $\phi_{(\bA^\tau)_l}:\V \to \V_{(\bA^\tau)_l}=_{\bA^\tau}\!\!\V$ is comonadic,  and
  the composite $\overline{\chi}$ in \ref{th.1.Alg}{\rm(d)}
  is an isomorphism;
\item $A$ is finite with right dual
 $(A^\sharp,{\rm db}':I\to A^\sharp\ot A,{\rm ev}':A\ot A^\sharp \to I)$,
 the functor $-\ot A:\V\to \V$ is monadic,   and
 the composite $\overline{\chi}_1$ in \ref{th.1.Alg.rev}{\rm(c)}
 is an isomorphism;
\item the right ${^e\!\mathcal{A}}$-module $(A, m \cdot ( m^\tau\ot A))$
 is faithfully Galois;
\item $\bA$ is a right Azumaya algebra.
\end{blist}
\end{thm}
\begin{proof}
In view of Proposition \ref{finite.reg} and Remark \ref{finite.reg-rem}, (a), (b), and (c) are equivalent by \ref{th.1.Alg}
and (d), (e), and (f) are equivalent by \ref{th.1.Alg.rev}.

 (c)$\LRa$(d) The composite $\overline\chi$ is the upper path and $\overline{\chi}_1$
is the lower path in the diagram
$$ \xymatrix{
 A\ot A \ar[d]_{\tau}\ar[r]^{\db'\cdot A\cdot A\qquad}& A^\sharp \ot A\ot A \ot A
\ar[d]^{A^\sharp\cdot A \cdot\tau}
\ar[r]^{A^\sharp  \cdot m\cdot A} &
 A^\sharp  \ot A\ot A    \ar[r]^{\quad A^\sharp  \cdot \tau }& A^\sharp \ot A \ot A
\ar[r]^{\quad A^\sharp \cdot m } &
       A^\sharp \ot A \\
 A\ot A \ar[r]_{\db'\cdot A\cdot A\qquad}&
   A^\sharp \ot A\ot A \ot A \ar[r]_{A^\sharp  \cdot \tau\cdot A} &
 A^\sharp \ot A\ot A\ot A \ar[ru]^{ A^\sharp  \cdot A \cdot m}
   \ar[r]_-{ A^\sharp\cdot m\cdot A}&
 A^\sharp \ot A \ot A, \ar[ru]_{A^\sharp  \cdot m } &
} $$
where $\tau=\tau_{A,A}$ and $\cdot =\ot$. The left square is commutative by naturality,
the pentagon is commutative since $\tau$ is a braiding,
and the parallelogram commutes by the associativity of $m$.
So the diagram is commutative and hence
$\overline{\chi}=\overline{\chi}_1\cdot \tau_{A,A}$,
that is, $\overline{\chi}$ is an isomorphism if and only if $\overline{\chi}_1$ is so.
Thus, in order to show that (c) and (d) are equivalent, it is enough to show that the functor
$\phi_{(\bA^\tau)_l}:\V \to {_{\bA^\tau}\V}$ is comonadic if and only if the functor
$-\ot A:\V\to \V$ is monadic. Since $\V$ is assumed to have equalisers and
coequalisers,
this follows from Lemma \ref{lem-mon} and Proposition \ref{finite.reg}.
\end{proof}

\begin{remark} \em A closer examination of the proof of the previous theorem shows that if
a braided monoidal category $\V$ admits
\begin{itemize}
  \item coequalisers, then any left Azumaya $\V$-algebra is right Azumaya,
  \item equalisers, then any right Azumaya $\V$-algebra is left Azumaya.
\end{itemize}
\end{remark}
\smallskip

In the setting of \ref{op-alg}, by Proposition \ref{prop.1}, the assignment
$$V \; \longmapsto \; ((A\ot V,\, A\ot A  \ot V \xra{m^\tau\ot V} A \ot V),
 A \ot A  \ot V \xra{m \ot V} A  \ot V)$$
yields the comparison functor
$K:\V \to (\V_{(\mathcal{A}^\tau)_l})_{\widehat{ \mathcal{A}_l }}=
(_{\mathcal{A}^\tau}\!\V)_{\widehat{ \mathcal{A}_l }}$.
\smallskip

Now assume  $A\ot -:\V \to \V$ to have a right adjoint functor $[A,-]:\V \to \V$
with unit $\eta^A:1\to [A,A\ot-]$.
Then there is a unique comonad structure $\widehat{[A,-]}$ on  $[A,-]$
(right adjoint to $ \mathcal{A}_l $, see \ref{mod-comod})
leading to  the  commutative diagram
\begin{equation}\label{A-adj}
\xymatrix{ \V \ar[rr]^-{K} \ar[drr]_{\phi_{(\mathcal{A}^\tau)_l}} &
& ({_{\mathcal{A}^\tau}\V})_{\widehat{ \mathcal{A}_l }}
\ar[d]^{U_{\widehat{ \mathcal{A}_l }}} \ar[rr]^{\Psi}&&
({_{\mathcal{A}^\tau}\V})^{\widehat{[A,-]}} \ar[d]^{U^{\widehat{[A,-]}}}\\
&& {_{\mathcal{A}^\tau}\V} \ar[rr]^= & & {_{\mathcal{A}^\tau}\V}, }
\end{equation}
where $\Psi=\Psi_{(\bA^\tau)_l}$. This is just the diagram (\ref{Azum.1b}) and
Theorem \ref{th.1} provides characterisations of left Azumaya algebras.

\begin{theorem} \label{th.1a.alg}
Let $\bA=(A,m,e)$ be an algebra in a braided monoidal category $(\V, \ot,I, \tau)$
and assume $A\ot-$ to have a right adjoint $[A,-]$ (see above).
Then the following are equivalent:
\begin{blist}
\item $ \mathcal{A}$ is left Azumaya;
\item the functor $\phi_{(\mathcal{A}^\tau)_l}:\V \to {_{\mathcal{A}^\tau}\!\V}$
is comonadic and,
 for any $V\in \V$, the composite:
  $$\begin{array}{rr}
\chi_V:A \ot A  \ot V \xra{(\eta^A)_{A  \ot A  \ot V}}& [A, A \ot A  \ot A  \ot V]
\xra{[A,m \ot A  \ot V]} [A, A \ot A  \ot A] \\[+1mm]
& \xra{[A, \tau_{A,A}\ot V]} [A, A  \ot A \ot V]\xra{[A, m \ot V]}[A, A  \ot V]
\end{array}$$ is an isomorphism;
\item  $A$ is finite, the functor $\phi_{(\mathcal{A}^\tau)_l}: \V \to {_{A^\tau}\!\V}$ is comonadic,
  and the composite
$$
\chi_I:A \ot A   \xra{(\eta^A)_{A  \ot A  }} [A, A \ot A  \ot A]
\xra{[A,m \ot A  ]}[A, A \ot A  ] \xra{[A, m^\tau]}[A, A]
$$
is an isomorphism.
\end{blist}
\end{theorem}
\begin{proof} (a)$\LRa$(b) follows by Theorem \ref{th.1}.

(a)$\LRa$(c)
Since $A$ turns out to be finite, there is a right dual
$(A^\sharp, \db', \ev')$ of $A$. Then
 $A^\sharp \ot \,-\, :\V \to \V$ and $[A, \,-\,] :\V \to \V$  are both right adjoint
to  $A \ot \,-\, :\V \to \V$,
and thus there is an isomorphism of functors $t: [A, \,-\,] \to A^\sharp \ot \,-\,$
inducing the commutative diagram
\begin{equation}\label{triangle}
\xymatrix{
V \ar[rrd]_{\db' \ot V}\ar[rr]^{(\eta^A)_V\quad} & & [A, A \ot V] \ar[d]^{t_{A \ot V}}\\
&& A^\sharp \ot A \ot V\,.}
\end{equation}
Rewriting the morphism $\bar\chi$ from \ref{th.1.Alg}(d) yields the morphism
$\chi_I$ in (c).
\end{proof}

Considering the symmetric situation we get:

\begin{theorem} \label{th.1a.alg.d}
Let $\bA=(A,m,e)$ be an algebra in a braided monoidal category $\V$
and assume $-\ot A$ to have a right adjoint $\{A,-\}$ with unit $\eta_A: 1 \to \{A, \,-\, \ot A\}$.
Then the following are equivalent:
\begin{blist}
\item $ \mathcal{A}$ is right Azumaya;
\item the functor $\phi_{(\mathcal{A}^\tau)_r}:\V \to \V{_{\mathcal{A}^\tau}}$ is comonadic and
 for any $V\in \V$, the composite
  $$\begin{array}{rr}
\chi'_V: V\ot A \ot A \xra{V \ot (\eta_A)_{A  \ot A }}& \{A,V \ot A \ot A  \ot A\}
\xra{\{A,V \ot m \ot A \}} \{A, V \ot A \ot A\} \\[+1mm]
& \xra{\{A, V \ot \tau_{A,A}\}} \{A, V \ot A  \ot A\}\xra{\{A, V \ot m \}}\{A, V  \ot A\}
\end{array}$$ is an isomorphism;
\item  $A$ is finite, the functor $\phi_{(\mathcal{A}^\tau)_r}: \V \to \V_{A^\tau}$
is comonadic, and the composite
$$
\chi'_I:A \ot A   \xra{(\eta_A)_{A  \ot A  }} \{A, A \ot A  \ot A\}
\xra{\{A,m \ot A  \}}\{A, A \ot A \} \xra{\{A, m^\tau\}}\{A, A\}
$$
is an isomorphism.
\end{blist}
\end{theorem}

\begin{remark} \em In \cite {VO-Z}, F. van Oystaeyen and Y. Zhang
defined {\em Azumaya algebras $\bA=(A,m,e)$ in $\V$} by
requiring $\bA$ to be left and right Azumaya in our sense (see \ref{azu-def}).
The preceding theorems \ref{th.1a.alg} and \ref{th.1a.alg.d}
correspond to the characterisation of these algebras in \cite[Theorem 3.1]{VO-Z}.
As shown in Theorem \ref{th.1.Alg.1}, if $\V$ admits equalisers and coequalisers it is
sufficient to require the Azumaya property on one side.
\end{remark}

Given an adjunction $(\db, \ve: V^* \dashv V)$ in $\V$, we know from \ref{duality} that $\mathcal{S}_{V, V^*}=V \ot V^*$ is
a $\V$-algebra. Moreover, it is easy to see that the morphism $V^* \ot V \ot V^* \xra{\ev \ot V^*} V^*$
defines a left $\mathcal{S}_{V, V^*}$-module structure on $V^*$, while the composite
$V\ot V^* \ot V \xra{V \ot \ev} V$ defines a right $\mathcal{S}_{V, V^*}$-module structure on $V$.

Recall from \cite{VO-Z} that an object $V\in \V$ with a left dual
$(V^*, \db, \ev)$ is \emph{right faithfully projective}
if the morphism $\overline{\ev}:V^* \ot_{\mathcal{S}_{V, V^*}} V \to I$
induced by $\ev: V^* \ot V \to I$  is an isomorphism. Dually, an object $V\in \V$ with a right dual $(V^\sharp, \db', \ev')$
is \emph{left faithfully projective} if the morphism
$\overline{\ev'}:V \ot_{\mathcal{S}_{V^\sharp, V}} V^\sharp \to I$ induced by
$\ev' :V \ot V^\sharp \to I$ is an isomorphism.

Since in a braided monoidal category an object is left faithfully projective
if and only if it is right faithfully projective (e.g. \cite[Proposition 4.14.]{F}), we do not have
to distinguish between left and right faithfully projective objects and we shall call them just faithfully projective.

\begin{theorem} \label{faithfully}
Let $(\V, \ot,I, \tau)$ be a braided closed monoidal category with equalisers and coequalisers. Let $\bA=(A,m,e)$ be
a $\V$-algebra such that the functor $A \ot \,-\,$ admits a right adjoint $[A,-]$ (and hence the functor
$\,-\,\ot A$ also admits a right adjoint $\{A,-\}$). Then the following are equivalent:
\begin{blist}
\item $ \mathcal{A}$ is left Azumaya;
\item $ \mathcal{A}$ is right Azumaya;
\item  $A$ is faithfully projective and the composite
$$
A \ot A   \xra{(\eta^A)_{A  \ot A  }} [A, A \ot A  \ot A]
\xra{[A,m \ot A  ]}[A, A \ot A  ] \xra{[A, m^\tau]}[A, A],
$$ where $\eta^A$ is the unit of the adjunction $A \ot \,-\, \dashv [A,-]$,
is an isomorphism.
\item  $A$ is faithfully projective and the composite
$$
A \ot A   \xra{(\eta_A)_{A  \ot A  }} \{A, A \ot A  \ot A\}
\xra{\{A,m \ot A  \}}\{A, A \ot A \} \xra{\{A, m^\tau\}}\{A, A\},
$$ where $\eta_A$ is the unit of the adjunction $\,-\,\ot A \dashv \{A,-\}$,
is an isomorphism.
\end{blist}
\end{theorem}
\begin{proof} That (a) and (b) are equivalent follows from Theorem \ref{th.1.Alg.1}.

(a)$\LRa$(c) Since in both cases $A$ is finite and thus the functor
$A \ot \,-\,:\V \to \V$
has both left and right adjoints, in view of Proposition \ref{finite.reg},
we get from Lemma \ref{lem-mon} that the functor
$\phi_{(\mathcal{A}^\tau)_l}:\V \to {_{\mathcal{A}^\tau}\!\V}$
is comonadic if and only if
 the functor $A \ot \,-\,:\V \to \V$ is conservative.
According to \cite[2.5.1, 2.5.2]{CF}, $A$ is
 faithfully projective if and only if $A$ is finite and the functor
$A \ot \,-\,:\V \to \V$ is conservative and hence
 the equivalence of (a) and (c) follows by Theorem \ref{th.1a.alg}.

 Similarly, one proves that (b) and (d) are equivalent.
\end{proof}

\begin{thm}\label{braided.mon} {\bf Braided closed monoidal categories.} \em
A braided  monoidal category $\V$ 
 is said to be  \emph{left closed} if each functor $V\ot -: \V  \to \V$ has a
right adjoint $[V,-] : \V  \to \V$,
we write $\eta^V,\ev^V: V\ot - \dashv [V,-]$.
$\V$ is called {\em right closed}
if each functor $-\ot V: \V \to \V$ has a right adjoint $\{V,-\} : \V \to \V $,
we write $\eta_V,\ev_V:-\ot V \dashv \{V,-\}$.
$\V$ being braided left closed implies that $\V$ is also right closed.
So assume $\V$ to be closed.
\smallskip

If $\bA$ is a $\V$-algebra, and $(V,\rho_V)\in {_\bA}\V$, then for any $X \in \V$,
$$(V \ot X,\, A \ot V \ot X \xra{\rho_V \ot X} V \ot X)\in {_\bA\V},$$
and  the
assignment $X \to (V \ot X, \rho_V \ot X) $ defines a functor $V \ot \,-\, : \V \to {_\bA\V}$.
When $\V$ admits equalisers, this functor has a right adjoint ${_\bA[V, -]}:{_\bA\V} \to \V,$
where for any $(W, \rho_W) \in {_\bA\V}$,  ${_\bA}[V, W]$ is defined to be the equalizer in $\V$ of
$$\xymatrix{
 [V, W] \ar@{->}@<0.5ex>[r]\ar@{->}@<-0.5ex>[r] &
 [A \ot V,W]},$$ where one of the morphisms is  $[\varrho_V, W]$, and
the other one is the composition
$$[V,W] \xra{(A \ot \,-\,)_{V,W}}[A\ot V,A\ot W] \xra{[A\ot V,\,\rho_W]}[A\ot V,W].$$

Symmetrically, for $V,W \in \V_\bA$, one defines $\{V,W\}_\bA$.
\medskip

The functor
$\ol{K}=\Psi K:\V \to (_{\mathcal{A}^\tau}\V)^{\widehat{[A,-]}}$
(in diagram (\ref{A-adj})) has as right adjoint
$\overline{R}:(_{\mathcal{A}^\tau}\V)^{\widehat{[A,-]}} \to \V$ (see \ref{right-adj}),
and since $\Psi$
is an isomorphism of categories, the composition
$\overline{R}\,\Psi$ is right adjoint to the functor
$K:\V \to (_{\mathcal{A}^\tau}\V)_{\widehat{ \mathcal{A}_l }}$.
Using now that $\mathcal{P}$ (see \ref{iso-P})
is an isomorphism of categories, we conclude that $\overline{R}\,\Psi \mathcal{P}$ is
right adjoint to the functor $\mathcal{P}^{-1}K:\V \to {_{\mathcal{A}^e}\V}$.
For any $(V,h)\in {_{\mathcal{A}^e}\V }$, we put
$${^\mathcal{A}V}:=\overline{R}\,\Psi\mathcal{P}(V,h).$$
Taking into account the description of the functors
$\mathcal{P}$, $\Psi$ and $\overline{R}$, one gets that $^\mathcal{A}V$
can be obtained as the equaliser of the diagram
$$\xymatrix{V\ar[r]^-{(\eta^A)_V}&
[A, A\ot V] \ar@{->}@<0.5ex>[rr]^-{[A,e\ot A   \ot V]}
\ar@{->}@<-0.5ex> [rr]_-{[A, A  \ot e \ot V]} &&
[A, A \otimes A   \ot V] \ar@{->}@<0.5ex>[r]^-{[A,h]}\ar@{->}@<-0.5ex>[r]_-{[A,h]} &
 [A,V]}.$$
\end{thm}

The functor $\mathcal{P}^{-1}K:\V \to {_{\mathcal{A}^e}\V }$ is just the functor
$A\ot - : \V \to {_{\mathcal{A}^e}\V }$ and admits as a right adjoint the functor
${_{\mathcal{A}^e}}[A,-]:{_{\mathcal{A}^e}\V } \to \V$ (see \ref{braided.mon}).
As right adjoints are unique up to isomorphism, we get an alternative proof for
B. Femi\'c's \cite[Proposition 3.3]{F}:

\begin{proposition}
Let $\V$ be a braided closed monoidal category with equalisers. For any $\V$-algebra $\bA$,
the functors
$${^{\mathcal{A}}(-)},\, {_{\mathcal{A}^e}}[A,-]:{_{\mathcal{A}^e}\V } \to \V$$
 are isomorphic.
\end{proposition}

This isomorphism allows for further characterisations of Azumaya algebras.

\begin{theorem} \label{invariants}
Let $\V$ be a braided closed monoidal category with equalisers. Then
a $\V$-algebra $\mathcal{A}=(A,m,e)$ is left  Azumaya if and only if
\begin{rlist}
    \item  the morphism $e: I \to A$ is a pure monomorphism, and
    \item  for any $(V,h)\in {_{\mathcal{A}^e}\V } $,
with the canonical inclusion $i_V : {^\mathcal{A}V} \to V$,
 the composite
$$A\ot {^\mathcal{A}V}\xra{A\, \ot\, i_V} A \ot V
    \xra{A\ot \, e \,\ot V} A \ot A   \ot V \xra{h} V$$
  is an isomorphism.
\end{rlist}
\end{theorem}
\begin{proof} The $\V$-algebra $\mathcal{A}$ is left Azumaya provided the functor
$\ol{K}_l:\V \to {_{\mathcal{A}^e}\!\V}$ is an equivalence of categories.
It follows from equation (\ref{counit}) that the composite
$$h \cdot (A\ot e\ot V) \cdot (A \ot i_V): \emph{A} \ot  {^\mathcal{A}}V\to V$$
is just the $\Psi\mathcal{P}(V,h)$-component of the counit of
$\ol{K}_l \dashv \overline{R}$ and hence is an isomorphism.
Moreover, by Proposition \ref{P.1.7}, the functor
$\phi_{(\mathcal{A}^\tau)_l}:\V \to {_\mathcal{A}\V}$
is comonadic, whence the morphism $e: I \to A$ is a pure monomorphism
(e.g. \cite[Theorem 2.1(2.(i))]{Me}). This proves one direction.

For the other direction we note that, under the conditions (i) and (ii),
the counit of the adjunction
$\mathcal{P}^{-1}\ol{K}_l \dashv \overline{R}\,\Psi\mathcal{P}$
(and hence also of the adjunction $\ol{K}_l=\Psi K \dashv \overline{R}$) is an
isomorphism and the functor $\phi_{(\mathcal{A}^\tau)_l}$ (and hence also $\ol{K}_l$)
is conservative (again \cite[Theorem 2.1(2.(i))]{Me}), implying
(as in the proof of Theorem \ref{th.2} (ii)) that $\ol{K}_l$
is an equivalence of categories.
\end{proof}

Symmetrically, for any $(V,h) \in \V_{^e\V}$ defining $V^\bA$ as the equaliser of the diagram
$$\xymatrix{V\ar[r]^-{(\eta_A)_V}&
\{A, V\ot A\} \ar@{->}@<0.5ex>[rr]^-{\{A,V \ot A \ot e\}}
\ar@{->}@<-0.5ex> [rr]_-{\{A,V\, \ot e\, \ot A\}} &&
[A, V \otimes A   \ot A\} \ar@{->}@<0.5ex>[r]^-{\{A,h\}}\ar@{->}@<-0.5ex>[r]_-{\{A,h\}} &
 \{A,V\}},$$ one has an isomorphism of functors
$(\,-\,)^\bA , \{A, \,-\,\}_{^e\!\bA}:\V_{^e\!\bA} \to \V$.

Dualising the previous theorem gives:

\begin{theorem} \label{invariants.d}
Let $\V$ be a braided closed monoidal category with equalisers. Then
a $\V$-algebra $\mathcal{A}=(A,m,e)$ is right Azumaya if and only if
\begin{rlist}
    \item  the morphism $e: I \to A$ is a pure monomorphism, and
    \item  for any $(V,h)\in {\V_{^e\V}} $,
with the canonical inclusion $i_V : {V^\bA} \to V$, the composite
$$V^\bA \ot V\xra{i_V \ot\, A} V \ot A
    \xra{V\ot \, e\, \ot A} V \ot A \ot A  \xra{h} V$$
  is an isomorphism.
\end{rlist}
\end{theorem}

\begin{definition} \em A $\V$-algebra $\bA$ is called \emph{left} (resp. \emph{right})
\emph{central} if there is an isomorphism $I \simeq {_{\mathcal{A}^e}}[A,-]$
(resp. $I \simeq \{A, \,-\,\}_{^e\!\bA}$). $\bA$ is called \emph{central} if it is
both left and right central.
\end{definition}

\begin{proposition}Let $\V$ be a braided closed monoidal category with equalisers. Then
\begin{rlist}
  \item   any left (resp. right) Azumaya algebra is left (resp. right) central;
  \item   if, in addition, $\V$ admits also coequalisers, then any $\V$ algebra that is Azumaya
  on either side is central.
\end{rlist}
\end{proposition}
\begin{proof}(i) follows by the Theorems \ref{invariants} and \ref{invariants.d}, while
(ii) follows from (i) and Theorem \ref{th.1.Alg.1}.
\end{proof}

\smallskip

Recall that for any $\V$algebra $\bA$,
an $\bA^e$-module $M$ is  {\em $U_{\bA^e}$-projective}
provided for morphisms $g:N\to L$ and $f:M\to L$ in ${_{\bA^e}\V}$ with $U_{\bA^e}(g)$ a
split epimorphism, there exists an $h:M\to N$ in ${_{\bA^e}\V}$ with $gh=f$.
This is the case if and only if $M$ is a retract of a (free) $\bA^e$-module ${\bA^e}\ot X$
with some $X\in \V$ (e.g. \cite{Sob}). This is applied in the characterisation of
separable algebras.


\begin{proposition}\label{sep.a} The following are equivalent for a $\V$-algebra
$\bA=(A,m,e)$:
\begin{blist}
  \item  $\mathcal{A}$ is a separable algebra;
  \item   $m: A\ot A \to A$ has a section $\xi: A\to A \ot A$ in $\V$ such that
$$(A \ot m)\cdot (\xi \ot A)=\xi \cdot m=(m \ot A)\cdot (A \ot \xi);$$
  \item the left $\mathcal{A}^e$-module $(A, m \cdot (A \ot m^\tau))$ is
    $_{\mathcal{A}^e}U$-projective;
  \item  the functor $_{\mathcal{A}^e}U : {_{\mathcal{A}^e}\V} \to \V$ is separable.
\end{blist}
\end{proposition}

\begin{proposition}\label{ten-A}  Consider $\V$-algebras $\bA$ and
 $\mathcal{B}$ such that the unit $e : I \to B$ of $\bB$ is a split monomorphism.
 If $\mathcal{A} \ot \mathcal{B}$ is separable in $\V$, then $\mathcal{A}$ is
also separable in $\V$.
\end{proposition}
\begin{proof} Since $I$ is a retract of $B$ in $\V$, $A$ is a retract of $A\ot B$ in
$_{\mathcal{A}^e}\V$.
Since $A \ot B$ is assumed to be separable in $\V$, $A\ot B$ is a retract of
$(A \ot B)^e $ in $_{(A\ot B)^e}\V$, and hence also in $_{A^e}\V$. Thus $A$ is a retract of
$A^e \ot B^e \simeq (A\ot B)^e$ in $_{A^e}\V$. Since $A^e \ot B^e=\phi_{\mathcal{A}^e}(B^e)$,
it follows 
that $A^e \ot B^e$ is  $_{A^e}U$-projective, and since
 retracts of a $_{A^e}U$-projectives are $_{A^e}U$-projective,
$A$ is $_{A^e}U$-projective and $\bA$ is separable by Proposition \ref{sep.a}.
\end{proof}

Following  \cite{P}, a finite object $V$ in $\V$ is said to be a \emph{progenerator}
if the counit morphism $\ev  :V^* \ot V \to I$ is a split epimorphism.
The following list describes some of its properties.

\begin{proposition}\label{prog} Assume $\V$ to admit equalisers and coequalisers.
For an algebra $\mathcal{A}=(A, m,e)$ in $\V$
with $A$ admitting a left adjoint $(V^*, \db, \ev)$, consider the following statements:
\begin{zlist}
  \item $A$ is a progenerator;
   \item the morphism $ {\db}:I \to A \ot A^*$ is a split monomorphism;
  \item the functor $A \ot -: \V \to \V$ is separable;
  \item the unit morphism $e:I \to A$ is a split monomorphism;
  \item the functor $A \ot -: \V \to \V$ is conservative (monadic, comonadic);
  \item $A \ot A^*$ is a separable $\V$-algebra.
  \end{zlist}

One always has
$(1)\LRa (2)\LRa (3)\LRa (4)\Ra (5)$  and $(1) \Ra (6)$.
\smallskip

If $I$ is projective (w.r.t. regular epimorphisms) in $\V$, then $(5)\Ra (1)$.
\end{proposition}

\begin{proof} Since $A$ is assumed to be admit a left adjoint $(V^*, \db, \ev)$, the functor
$A^* \ot -: \V \to \V$ is  left as well as right
adjoint to the functor $A \ot -: \V \to \V$.
For any $V \in \V$, the composite
$$V \xra{\db\ot V} A \ot A^* \ot V \xra{\tau^{-1}_{A^*,A}\ot V} A^* \ot A \ot V$$
is the $V$-component of the unit of the adjunction $A\ot - \dashv A^* \ot - :\V \to \V$,
 while
the morphism $ A^* \ot A\ot V \xra{\ev \ot V} V$
is the $V$-component of the counit of the adjunction
$A^*\ot - \dashv A \ot - :\V \to \V$.
 To say that $\db: I \to A \ot A^*$ (resp. $\ev: A^* \ot A \to I$
is a split monomorphism (resp. epimorphism) is to say that the unit (resp. counit)
of the adjunction
$A\ot \,-\,\dashv A^* \ot \,-\,$ (resp. $A^*\ot \,-\,\dashv A \ot \,-\,$)
is a split monomorphism (resp. epimorphism).
From the observations in \ref{sep-func}, one gets
$(1)\LRa (2)\LRa (3)$.

By Proposition \ref{finite.reg}, the properties listed in (5) are equivalent.

Since $\V$ admits equalisers, it is Cauchy complete, and the
implication $(3)\Ra (5)$ follows from \cite[Proposition 3.16]{Me}.

If $e:I \to A$ is a split monomorphism, then the natural transformation
$$e \ot -: 1_\V \to A \ot -$$ is a split monomorphism and applying Proposition \ref{rafael} to
the pair of functors $(A \ot \,-\,, 1_\V)$  gives that the functor $A \ot -: \V \to \V$
is separable, proving  $(4)\Ra (3)$.

If $A$ is a progenerator, then  $\ev  :A^* \ot A \to I$
has a splitting $\zeta:I \to A^* \ot A$. Consider the composite
$$\phi: A \xra{\zeta \ot A} A^* \ot A \ot A \xra{A^* \ot\, m} A^* \ot  A \xra{\ev } I.$$
We claim that $\phi \cdot e=1$. Indeed, we have
$$\begin{array}{rl}
 \ev  \cdot A^* \ot m \cdot \zeta \ot A \cdot e & =\;
  \ev   \cdot A^* \ot m \cdot   A^* \ot A \ot e \cdot \zeta =\ev   \cdot \zeta=1.
\end{array}
$$ The first equality holds by naturality, the second one since $e$ is the unit for the
$\V$-algebra $\mathcal{A}$,
and the third one since $\zeta$ is a splitting for $\ev  :A^* \ot A \to I$. Thus
$(2)$ implies $(4)$.

Now, if $A$ is again a progenerator, then the morphism $\ev  :A^* \ot A \to I$
has a splitting $\zeta:I \to A^* \ot A$,
and direct inspection shows that the morphism
$$\xi=A \ot \zeta \ot A^*: A \ot A^* \to A \ot A^*\ot A \ot A^*$$ is a splitting for
the multiplication $A\ot \ev \ot A^*$ of the $\V$-algebra
$\bA \ot \bA^*$  satisfying condition (b) of Proposition \ref{sep.a}.
Thus $\bA\ot \bA^*$ is a separable $\V$-algebra, proving
the implication $(2)\Ra (6)$.

Finally, suppose that $I$ is projective (w.r.t. regular epimorphisms) in $\V$ and that
the functor $A \ot - :\V \to \V$ is monadic. Then, by \cite[Theorem 2.4]{KP}, each component of the counit
of the adjunction $A^* \ot - \dashv A \ot -$ is a regular epimorphism. Since $\ev : A^* \ot A \to I$
is the $I$-component of the counit, $\ev $ is a regular epimorphism, and hence splits, since $I$ is
assumed to be projective w.r.t. regular epimorphisms. Thus $A$ is a progenerator. This proves the implication $(5) \Ra (1).$
\end{proof}

\begin{theorem}\label{Azu-sep} Let $\V$ be a braided monoidal category with equalisers
and coequalisers.
For an algebra $\bA=(A,m,e)$ in $\V$, the following are equivalent:
\begin{blist}
  \item $\mathcal{A}$ is a separable left Azumaya $\V$-algebra;
  \item $A$ is a progenerator  and the morphism $\overline{\chi}_0:A\ot A \to A\ot A^* $
       in \ref{th.1.Alg}(c) is an isomorphism between the $\V$-algebras $\mathcal{A}^e$ and $\mathcal{S}_{A,A^*}$;
   \item $e:I \to A$ is a split monomorphism and
        $(A, m \cdot (A \ot m^\tau))\in _{\mathcal{A}^e}\!\!\V$ is a Galois module.
\end{blist}
\end{theorem}
\begin{proof}  (a)$\LRa$(c) In view of Proposition \ref{sep.a}, this is a special case
of \ref{Azu-mon-comon}.

(b)$\LRa$(c) is an easy consequence of Proposition \ref{prog} and Theorem \ref{th.1.Alg}.
\end{proof}

To bring back our general theory to the starting point,
let $R$ be a commutative ring with identity  and $\M_R$
the category of $R$-modules.
Then for any $M, N\in \M_R$, there is the canonical twist map
$\tau_{M,N}: M\ot_R N\to N\ot_R M$.
Putting  $[M,N]:=\Hom_R(M,N)$, then $(\M_R, -\ot_R -, R, [-,-], \tau)$ is
a symmetric monoidal closed category.
We have the canonical adjunction $\eta^M,\ve^M: M\ot_R-\dashv [M,-]$.

\begin{thm}{\bf Algebras in $\M_R$.} \em
 For any $R$-algebra  $\bA =(A,m,e)$,
$\tau_{A,A}:A\ot_R A\to A\ot_R A$ is an invertible (involutive) BD-law
allowing for the definition of the (opposite) algebra $\bA^\tau=(A,m\cdot\tau, e)$.
The monad $A\ot_R-$ is Azumaya provided the functor
$$K : \M_R \to {_{A^e}\M},$$
$$ M \; \longmapsto \; ((A \ot_R M,\, A \ot_R A \ot_R A \ot_R M \xra{A \ot_R m^\tau \ot_R M} A \ot_R A \ot_R M
\xra{m\ot_R M} A \ot_R M),$$
is an equivalence of categories. Obviously this holds if and only if $A$ is an
Azumaya $R$-algebra in the usual sense.
 We have the commutative diagram
\begin{equation}\label{com.diag}
\xymatrix{\M_R  \ar[rr]^-{K} \ar[drr]_{\phi_{(A^\tau)_l}= A^\tau  \ot_R  -} && {_{{A}^e}\M } \ar[rr]^-{\Psi}
\ar[d]^{(e\ot_R A^\tau)^*} && (_{ A^\tau}\M)^{\widehat{[A,-] }}
\ar[d]^{U^{\widehat{[A,-]}}}  \\
&& {_{A^\tau}\M} \ar[rr]_= && {_{A^\tau}\M} }
\end{equation} where $(e\ot_R A^\tau  )^*$ is the restriction of scalars functor induced by
the ring morphism
$e\ot_R A^\tau: A^\tau  \to A\ot_R A^\tau$.

It is not hard to see that, for any $(M,h) \in {_{A^\tau}\M}$, the $(M,h)$-component
$t_{(M,h)}:A  \ot_R  M \to [A, M]$ of the comonad morphism
$t: \phi_{(A^\tau)_l}U_{(A^\tau)_l} \to \widehat{[A,-] }$
corresponding to the functor $\ol{K}= \Psi K$, takes any element $ {a}\ot_R m$ to
the map ${b} \mapsto h((ba) \ot_R m)$. Thus, writing $a\cdot m$ for $h(a\ot_R m)$,
one has for  $a,b \in A$  and  $m \in M$,
$$t_{(M,h)} (a \ot_R  m)=({b} \mapsto (ba) \cdot m).$$
In particular, for any $N \in \M_R$,
$t_{\phi_{(A^\tau)_l}(N)} (a \ot_R b \ot_R  n)=({c} \mapsto (bca) \cdot n).$
\end{thm}

Since the canonical morphism $i: R \to A$ factorises through the center of $A$,
it follows from   of \cite[Theorem 8.11]{Me} that the functor $A\ot_R - :\M_R \to {_A\M}$
(and hence also $A^\tau  \ot_R - :\M_R \to {_{A^\tau}\M}$) is comonadic
if and only if $i$ is a pure morphism of $R$-modules. Applying Theorem \ref{th.1a.alg} and using that
 $K$ is an equivalence of categories if and only if $\ol{K}= \Psi K$ is so, we get
several characterisations of Azumaya $R$-algebra.

\begin{theorem}\label{azu-alg} An $R$-algebra $A$ is an Azumaya $R$-algebra if and only if
  the canonical morphism $i: R \to A$ is a pure morphism of $R$-modules, and one
  of the following holds:
 \begin{blist}
 \item  for any  $M \in {_{A^\tau}\M} $, there is an isomorphism
 $$A \ot_R  M \to [A, M], \quad  {a} \ot_R  m
    \mapsto [b \mapsto (ba)\cdot m];$$
 \item  for any  $N \in \M_R$, there is an isomorphism
       $$A \ot_R  A \ot_R  N \to [A, A \ot_R  N],\quad
         {a} \ot_R   {b} \ot_R  n \mapsto [c \mapsto bca \ot_R n];$$
 \item  $A_R$ is finitely generated projective  and there is an isomorphism
$$A\ot_R A \to [A, A],\quad  {a} \ot_R  b \mapsto [c \mapsto bca ];$$

\item  for any   $(A,A)$-bimodule $M$, the evaluation map is an isomorphism
    $$A \ot_R M^A\to M,\quad a \ot_R   m \mapsto a \cdot m.$$
\end{blist}
\end{theorem}
\begin{proof}
  (a) follows by Theorem \ref{th.1};
(b) and (c) are derived from Theorem \ref{th.1a.alg}.

(c)  An $R$-module is finite in the monoidal category $\M_R$ if and only if it is finitely
generated and projective over $R$ and  Theorem \ref{finite} applies.

(d) is a translation of Theorem \ref{invariants} into the present context.
\end{proof}

For a (von Neumann) regular ring $R$, $i:R \to A$ is always a pure $R$-module morphism,
and hence over such rings the (equivalent) properties (a) to (d) are sufficient
to characterise Azumaya algebras.

\section{Azumaya coalgebras in  braided monoidal categories}\label{azu-braid-co}

Throughout $(\V, \ot, I, \tau)$ will denote a strict monoidal braided category.
The definition of coalgebras $\mathcal C=(C,\Delta,\ve)$ in $\V$ was recalled in
\ref{Coalg.comod}.

\begin{thm}\label{bicom}{\bf The coalgebra $\bC^e$.} \em
 Let $\bC$ be a $\V$-coalgebra.
The braiding $\tau_{C,C}:C\ot C\to C\ot C$
provides a BD-law  allowing for the definition of the opposite coalgebra
$\mathcal C^\tau=(C^\tau,\Delta^\tau=\tau_{C,C} \cdot \Delta, \ve^\tau=\ve)$
and a coalgebra
 $$\bC^e:=(C \ot C^\tau, (C\ot \tau\ot C^\tau) (\Delta \ot\Delta^\tau),\ve\ot \ve ) .$$

 Writing $\tau: \bC_l (\bC^\tau)_l \to (\bC^\tau)_l \bC_l$ for the induced distributive law
of the comonad $\bC_l$ over the comonad  $(\bC^\tau)_l$, we have an isomorphism of categories
$\V^{(\bC^\tau)_l \bC_l}\simeq \V^{(\mathcal C^e)_l} ={^{\mathcal C^e}}\V$.
\end{thm}

\begin{thm}\label{def-azu-com} {\bf Definition.} \em (see \ref{def-Azu-co})
A $\V$-coalgebra $\bC$ is said to be  {\em left Azumaya} provided the comonad
$\bC_l=C \ot -:\V \to \V$ is Azumaya, i.e. the comparison
functor
$$\overline{K}_\tau :\V \to {{^{\mathcal C^e}}\!\V}, \quad
V\;\longmapsto \;(C \ot V,\, C \ot V\xra{\Delta \ot V}
C\ot C \ot V \xra{C\ot\Delta ^\tau\ot V} C\ot C \ot C \ot V),$$
is an equivalence of categories. It fits into the commutative diagram
\begin{equation}\label{comp.f.com}
\xymatrix{
 {\V}\ar[r]^-{\overline{K}_\tau}\ar[dr]_{C\ot-} &
 **[r]{^{\bC^e} \V}={\V^{(\mathcal C^e)_l}} \ar[d]^{^{\mathcal C^e}\!U} \\
  & {\V} \,.}
\end{equation}

$\bC$ is said to be  {\em right Azumaya} if the corresponding conditions for $\bC_r=-\ot C$
are satisfied. Similar to \ref{finite} we have:
\end{thm}

\begin{proposition}\label{finite.d}
Let $\bC=(C,\Delta,\ve)$ be a coalgebra in a braided monoidal category $\V$.
 If $\bC$ is left Azumaya, then $C$ is finite in $\V$.
\end{proposition}
\begin{proof} Suppose that a $\V$-coalgebra $\bC$ is left Azumaya. Then the functor
$C \ot \,-\, :\V \to \V$ admits a right adjoint $[C,\,-\,]:\V \to \V$ by Proposition \ref{l-r-adj-co}.
Write $\vartheta$ for the composite $(C \ot \Delta^\tau)\cdot \Delta: C \to C \ot C \ot C$. Then for any
$V \in \V$, $\overline{K}_\tau(V)=(C \ot V, \vartheta \ot V)$ and thus the $V$-component of the
left ${\bC^e}$-comodule structure on the functor $C \ot \,-\, :\V \to \V$, induced
by the commutative diagram (\ref{comp.f.com}), is the morphism $\vartheta \ot V: C \ot V \to C \ot C \ot C \ot V$.
From \ref{Gal} we then see that the $V$-component $t_V$ of the comonad morphism induced by the above diagram is
the composite $$C \ot [C,V] \xra{\vartheta \ot [C,V]}C \ot C \ot C \ot [C,V]\xra{C \ot C \ot (\ev^C)_V} C \ot C \ot V,$$
where $\ev^C$ is the counit of the adjunction $C \ot \,-\, \dashv [C,\,-\,]$.

Next, let $\sigma_V: [C,I] \ot V \to [C,V]$ be the transpose of the morphism $(\ev^C)_I \ot V :C\ot [C,I] \ot V \to V$
and consider the following diagram
$$
\xymatrix{C\ot [C,I] \ot V \ar[rr]^-{\vartheta \ot [C,I] \ot V} \ar[d]_{C \ot \sigma_V}&&
 C \ot C \ot C\ot [C,I] \ot V
\ar[d]_{C \ot C \ot C\ot \sigma_V}\ar[rrd]^{C \ot C \ot (\ev^C)_I \ot V}&&\\
C\ot [C,V] \ar[rr]_-{\vartheta \ot [C,V]}&& C \ot C \ot C\ot [C,V] \ar[rr]_-{C\ot C \ot (\ev^C)_V}&& C\ot C \ot V\,.}
$$ In this diagram the rectangle is commutative by naturality of composition.
Since $\sigma_V$ is the transpose of the morphism $(\ev^C)_I \ot V$, the transpose
of $\sigma_V$ -- which is the composite $C \ot [C,I] \ot V \xra{C \ot \sigma_V} C \ot [C,V] \xra{(\ev^C)_V} V$ --
is $(\ev^C)_I \ot V$. Hence the triangle in the diagram is also commutative.
Now, since
$$(C \ot C \ot (\ev^C)_I \ot V)\cdot (\vartheta \ot [C,I] \ot V)=t_I \ot V,$$
it follows
from commutativity of the diagram that $t_I \ot V=t_V \cdot (C \ot \sigma_V)$, and since
$\bC$ is assumed to be left Azumaya, both $t_I$ and $t_V$ are isomorphisms, one concludes that $C \ot \sigma_V$
is an isomorphism. Moreover, the functor $C \ot \,-\, :\V \to \V$ is comonadic, hence conservative. It follows that
$\sigma_V: [C,I] \ot V \to [C,V]$ is an isomorphism for all $V \in \V$. Thus the functor $[C,I]\ot \,-\,:\V \to \V$
is also right adjoint to the functor $C \ot \,-\, :\V \to \V$. It is now easy to see that $[C,I]$ is right adjoint to $C$.
\end{proof}

Theorem \ref{th.1.co-Monad} provides a first characterisation of left Azumaya coalgebras.

\begin{theorem}\label{th.1.Clg} 
For a $\V$-coalgebra $\bC=(C,\Delta,\ve)$, the following are equivalent:
\begin{blist}
  \item  $\bC$ is a left Azumaya $\V$-coalgebra;
  \item the functor $C \ot -:\V \to \V$ is comonadic and the left $(\mathcal C^e)_l$-comodule
  structure on it, induced by the commutative diagram (\ref{comp.f.com}), is Galois;
  \item  {\rm (i)} $C$ is finite with right dual $(C^\sharp,  {\db'}: I \to C^\sharp \ot C, {\ev'}:C \ot C^\sharp \to I)$,
  the functor $C \ot \,-\,:\V \to \V$ is comonadic and \\[+1mm]
  {\rm (ii)}   the composite $\overline{\chi}_0:$
$$ C\ot C^\sharp \xra{\Delta \ot C^\sharp} C \ot C \ot C^\sharp
\xra{C \ot \Delta \ot C^\sharp}
 C \ot C \ot C \ot C^\sharp   \xra{C \ot  \tau \ot
C^\sharp} C \ot C \ot C \ot C^\sharp \xra{C \ot C \ot \ev'} C\ot C $$
is an isomorphism (between the $\V$-coalgebras $\mathfrak{S}_{C,C^\sharp}$ and ${\mathcal C}^e$);
  \item  {\rm (i)} $C$ is finite with left dual $(C^*,  {\db}: I \to C \ot C^*, {\ev}:C^* \ot C \to I)$,
  the functor $\phi_{(\bC^\tau)_l} :\V \to \V^{(\bC^\tau)_l}=\,^{\bC^\tau}\V$ is monadic and \\[+1mm]
  {\rm (ii)}   the composite $\overline{\chi}:$
  $$ C^*\ot C \xra{C^* \ot \Delta} C^* \ot C \ot C \xra{C^* \ot \tau}
 C^* \ot C \ot C   \xra{C^* \ot  \Delta \ot C} C^* \ot C \ot C \ot C
 \xra{\ev \ot C \ot C } C\ot C $$
is an isomorphism.
\end{blist}
\end{theorem}
\begin{proof}(a) and (b) are equivalent by Theorem \ref{th.1.co-Monad}.

The equivalences (a) $\LRa$ (c) and (a) $\LRa$ (d) follow from Proposition \ref{finite.d}
by dualising the proofs of the corresponding equivalences in Theorem \ref{th.1.Alg}.
\end{proof}

Similarly, writing out Theorem \ref{th.3} and
the dual form of Theorem \ref{th.1.Alg.rev} yields conditions for
right Azumaya coalgebras $\bC$, that is, making $\bC_r=-\ot C$ an Azumaya comonad. Dualising
Theorem \ref{th.1.Alg.1} gives:

\begin{theorem}\label{th.1.Alg.1.D}  Let $\bC=(C,\Delta,\ve)$ be a $\V$-coalgebra
in a braided monoidal category $\V$ with equalisers and coequalisers.
 Then the following are equivalent:
\begin{blist}
\item  $\bC$ is a left Azumaya coalgebra;
\item the left $\mathcal{C}^e$-comodule $(C, (C \ot \Delta^\tau) \cdot \Delta)$ is cofaithfully Galois;
\item there is an adjunction ${\rm db}', {\rm ev}':C \dashv C^\sharp$,
  the functor ${\,-\,\ot \,C}:\V\to \V$ is comonadic,  and
  the composite $\overline{\chi}$ in \ref{th.1.Clg}{\rm(c)}
  is an isomorphism;
\item the right $^e\mathcal{C}$-comodule $(C, (\Delta^\tau \ot C) \cdot \Delta)$ is
cofaithfully Galois;
\item $\bC$ is a right Azumaya coalgebra.
\end{blist}
\end{theorem}

Under suitable assumptions, the base category $\V$ may be replaced by a
comodule category over a cocommutative coalgebra. For this we consider the

\begin{thm}\label{C-Azumaya}{\bf Cotensor product.} \em
Suppose now that $\V = (\V,\ot, I, \tau)$ is a braided monoidal category with equalisers and
$\bD=(D,\Delta_\bD,\ve_\bD)$ is a coalgebra in $\V$. If
$((V,\rho^V) \in {^{\bD}\V}$ and $(W,\varrho^W)\in {\V^{\bD}}$, then their \emph{cotensor product}
(\emph{over} $\bD$) is the object part of the   equaliser
$${\xymatrix{V \ot^D W \ar[r]^{i_{V, W}} &
 V\otimes W \ar@{->}@<0.5ex>[rr]^-{\varrho^V \ot  W} \ar@
{->}@<-0.5ex> [rr]_-{V \ot  \rho^W}&& V  \ot D \ot W &}}$$

Suppose in addition that either
\begin{itemize}
\item[-] for any $V\in \V$, $V \ot -:\V \to \V$
and $-\ot V:\V \to \V$  preserve equalisers, or
\item[-] $\V$ is Cauchy complete and $\bD$ is coseparable.
\end{itemize}
Each of these condition guarantee that for  $V, W, X \in {^{\bD}\V^{\bD}}$,
\begin{itemize}
  \item $V \ot^D W \in {^{\bD}\V^{\bD}}$;
  \item  the canonical morphism (induced by the associativity of the tensor product)
$$(V \ot^D W) \ot^D V \to V \ot^D (W \ot^D X)$$
  is an isomorphism in $^{\bD}\V^{\bD}$;
  \item $(^{\bD}\V^{\bD}, -\ot^D -, D, \widetilde{\tau})$, where $\widetilde{\tau}$ is the restriction of $\tau$,
  is a braided monoidal category.
\end{itemize}

When $\bD$ is cocommutative (i.e. $\tau_{D,D}\cdot \Delta=\Delta$), then for any $(V,\rho^V)\in {^{\bD}\V}$, the composite
$\rho_1^V= \tau^{-1}_{D, V}\cdot \rho^V: V  \to V\ot D$, defines  a right $\bD$-comodule structure on $V$. Conversely, if
$(W,\varrho^W)\in {\V^{\bD}}$, then $\varrho_1^W=\tau_{W, D}\cdot \varrho^W : W \to D \ot W$ defines
a left $\bD$-comodule structure on $W$. These two constructions establish an isomorphism between
$^{\bD}\V$ and $\V^{\bD}$, and thus we do not have to
distinguish between left and right $\bD$-comodules. In this case, the tensor product of two $\bD$-comodules
is another $\bD$-comodule, and cotensoring over $\bD$ makes $^{\bD}\V$ (as well as $\V^{\bD}$) a braided
monoidal category with unit $\bD$.
\end{thm}

\begin{thm}\label{C-Azumaya.D}{\bf $\bD$-coalgebras.} \em
Consider $\V$-coalgebras $\bC=(C, \Delta_\bC, \ve_\bC)$  and  $\bD=(D,\Delta_\bD,\ve_\bD)$
with $\bD$ cocommutative.
A coalgebra morphism $\gamma:\bC\to \bD$ is called {\em cocentral}
provided the diagram
$$\xymatrix{ C \ar[r]^{\Delta_\bC\quad} \ar[d]_{\Delta_\bC} & C\ot C \ar[r]^{C \ot \gamma } &
      C\ot D \ar[d]^{\tau_{C,D}}\\
C\ot  C \ar[rr]_{\gamma \ot C} && D\ot C } $$
is commutative. When this is the case, $(\bC,\gamma)$ is called a $\bD$-coalgebra.


To specify a $^{\bD}\V$-coalgebra structure on
an object $C \in \V$ is to give $C$ a $\bD$-coalgebra
structure $(\bC=(C, \Delta_\bC, \ve_\bC), \gamma)$.
Indeed, if $\gamma:\bC\to \bD$ is a cocentral morphism,
$\bC$ can be viewed as an object of $^{\bD}\V$ (and $\V^{\bD}$) via
$$C \xra{\Delta_C} C \ot C \xra{\gamma \ot C} D \ot C, \quad
 ( C \xra{\Delta_C} C \ot C \xra{C \ot \gamma} C \ot D \xra{\tau_{C, D}} D \ot C),$$
and $\Delta_C$ factors through the $i_{C, C}:C \ot_D C \to C \otimes C$ by some
(unique) morphism $\Delta'_\bC :C \to C \ot_D C$, that is
${\Delta_\bC} = i_{C, C}\cdot {\Delta'_\bC}$.

\smallskip

 The triple $\bC_\bD =(C,
\Delta'_\bC, \gamma)$ is a coalgebra in the braided monoidal category $^{\bD}\V$.
\smallskip

Conversely, any $^{\bD}\V$-coalgebra,
$(C, \Delta'_\bC: C \to C \ot^D C, \ve_\bC: C \to D)$
induces a $\V$-coalgebra
$$\bC=(C, C \xra{\Delta'_\bC}C \ot^D C \xra{i_{C,C}} C \ot C,
C \xra{\ve_\bC} D \xra{\ve_D} I),$$
  and the pair $(\bC, \ve_\bC)$ is a $\bD$-coalgebra.

Related to any $\V$-coalgebra morphisms $\gamma: \bC\to \bD$, there is the
\emph{corestriction} functor
$$(\, - \,)_\gamma: {^\bC\V}\to {^\bD\V}, \quad (V,\varrho^V)\;\mapsto \;
      (V,(V\ot \gamma)\cdot \varrho^V),$$
and usually one writes $(V)_\gamma=V$.
If the category ${^\bC\V}$ admits equalisers, then one has  the {\em coinduction functor}
 $$ C\ot^D- : {^\bD\V}\to {^\bC\V}, \quad W\mapsto (C\ot^DW, \Delta_\bC \ot^DW),$$
defining an adjunction
$$(\,-\,)_\gamma \dashv C\ot^D-:{^\bD\V}\to {^\bC\V}.$$

Considering $\bC$ as a $(\bD,\bC)$-bicomodule by
$C\stackrel{\Delta}{\lra} C\ot_R C\stackrel{\gamma \ot C}{\lra} D\ot_R C$,
the corestriction functor is isomorphic to $C\ot^C-: {^\bC\V}\to {^\bD\V}$.

If $(\bC, \gamma)$ is a $\bD$-coalgebra, then the category
$^{\bC_\bD}(^{\bD}\V)$ can be identified with the category $^\bC\V$  and,
modulo this identification,
the functor $$C_\bD \ot^D \,-\,: {^{\bD} \V} \to {^{\bC_\bD}} ({^{\bD}\V})$$
corresponds to the coinduction functor
$C \ot^D \,-\, : {^{\bD}\V} \to {^{\bC}\V}.$
\end{thm}

  \begin{thm}\label{def-azu-D} {\bf Azumaya $\bD$-coalgebras.} \em
Let $\bD$ be a cocommutative $\V$-coalgebra. Then a $\bD$-coalgebra $\bC=(C,\Delta_\bC,\ve_\bC)$
is said to be {\em left Azumaya} provided the comonad
$$\bC_l=C \ot^D -: {^\bD\V} \to {^\bD\V}$$ is Azumaya, i.e. (see \ref{def-Azu-co}),
the comparison functor
$\overline{K}_{\widetilde{\tau}} :{^\bD\V}\to {^{C\ot^D C^{\widetilde{\tau}}}\V}$
defined by
$$V\;\longmapsto \;(C \ot^D V,\, C\ot^D V \xra{\Delta_\bC \ot^D V}
C \ot^D C \ot^D V \xra{C \ot^D \Delta_\bC^{\widetilde{\tau}}\ot^D V} C\ot^D C \ot^D V)$$
is an equivalence of categories.
\end{thm}

In this setting, the results from Section \ref{co-azu} - and also
 specializing Theorem \ref{th.1.Clg} - yield  various characterisations of
Azumaya $\bD$-coalgebras.

\smallskip

Now let $R$ be again a commutative ring with identity  and $\M_R$
the category of $R$-modules. As an additional notion of interest the dual
algebra of a coalgebra comes in.

\begin{thm}\label{coalg-mod}{\bf Coalgebras in $\M_R$.} \em
 An $R$-coalgebra $\mathcal C=(C,\Delta,\ve)$ consists of an $R$-module $C$ with the
$R$-linear maps multiplication $\Delta:C\to C\ot_R C$ and counit $\ve:C\to R$
subject to coassociativity and counitality conditions.
$C\ot_R-:\M_R\to \M_R$ is a comonad and
it is customary to write $^C\M:= \M_R^{C\ot-}$ for the category of left
$\bC$-comodules.
We write $\Hom^C(M,N)$ for the comodule morphisms between $M,N\in {^C\M}$.
In general, $^C\M$ need not be a Grothendieck category unless $C_R$ is a flat $R$-module
(e.g. \cite[3.14]{BW}).
\end{thm}

The dual module $C^*=\Hom_R(C,R)$ has an $R$-algebra structure by defining
 for $f,g\in C^*$, $f*g= (g\ot f)\cdot \Delta$ 
(definition opposite to \cite[1.3]{BW})
and there is a faithful functor
 $$\Phi:{^C\M} \to {_{C^*}\M}, \quad (M,\varrho) \;\mapsto\;
   C^*\ot_RM \xra{C^*\ot \varrho} C^*\ot_R C \ot M \xra{\ev\ot M} M ,$$
where $\ev$ denotes the evaluation map.
The functor $\Phi$ is full if and only if for any $N\in \M_R$,
$$\alpha_N: C\ot_R N \to \Hom_R(C^*,N), \quad c\ot n \mapsto [f\mapsto f(c)n],$$
is injective and this is equivalent to
 $C_R$ being locally projective ($\alpha$-condition, e.g. \cite[4.2]{BW}).
In this case ${^C\M}$ can be identified with the full subcategory
$\sigma[{_{C^*}C}] \subset {_{C^*}\M}$ subgenerated by $C$ as  $C^*$-module.

The $R$-module structure of $C$ is of considerable relevance for the
related constructions and for convenience we recall:

\begin{thm}\label{f.g.p} {\bf Remark.} \em For $C_R$ the following are equivalent:
 \begin{blist}
\item $C_R$ is finitely generated and projective;
\item $C\ot_R-:\M_R\to\M_R$ has a left adjoint;
\item  $\Hom_R(C,-):\M_R\to\M_R$ has a right adjoint;
\item
  $C^*\ot_R - \to \Hom_R(C,-), \; f\ot_R - \mapsto (c\mapsto f(c)\cdot -),$
is a (monad) isomorphism;
\item
  $C\ot_R - \to \Hom_R(C^*,-), \; c\ot_R - \mapsto (f\mapsto f(c)\cdot -),$
is a (comonad) isomorphism;
\item $\Phi:{^C\M} \to {_{C^*}\M}$ is a category isomorphism.
\end{blist}
If this holds, there is an algebra anti-isomorphism $\End_R(C) \simeq \End_R(C^*)$
and we denote the canonical adjunction by $\eta^C,\ve^C: C\ot_R-\dashv C^*\ot_R-$.
\end{thm}

\begin{thm}\label{bico}{\bf The coalgebra $\bC^e$.} \em
As in \ref{bicom}, the twist map $\tau_{C,C}:C\ot_R C\to C\ot_R C$
provides an (involutive) BD-law  allowing for the
definition of the opposite coalgebra $\mathcal C^\tau=(C^\tau,\Delta^\tau, \ve^\tau)$
and a coalgebra
 $$\bC^e:=(C \ot_R C^\tau, (C\ot_R \tau\ot_R C^\tau) (\Delta \ot_R\Delta^\tau),\ve\ot_R \ve ) .$$
The category  $^{\bC^e}\M$ of left ${\bC^e}$-comodules is just the category of
$(C,C)$-bicomodules (e.g. \cite{Lar}, \cite[3.26]{BW}).
A direct verification shows that the endomorphism algebra of $C$ as $\bC^e$-comodule
is just the center of $C^*$, that is,
   $$Z(C^*)=\Hom^{\bC^e}(C,C)\subset {^{\bC}\Hom(C,C)}\simeq C^*.$$

If $C_R$ is locally projective, an easy argument shows that $C\ot_R C$
is also locally projective as $R$-module and then $^{\bC^e}\M$ is a full subcategory
of ${_{(\bC^e)^*}\M}$.
\end{thm}

\begin{thm}\label{def-azu-co} {\bf Definition.} \em
An $R$-coalgebra $\bC$ is said to be an {\em Azumaya coalgebra} provided the comonad
$\bG=C \ot _R -:\M_R \to \M_R$ is Azumaya, i.e. (see \ref{def-Azu-co}), the comparison
functor
$K :\M_R \to {^{\bC^e}\M}$ defined by
$$M\;\longmapsto \;(C \ot_R M,\, C\ot_R M\xra{\Delta \ot_R M}
C\ot_R C \ot_R M \xra{C \ot \Delta^\tau \ot_R M} C \ot C\ot_R C \ot_R M)$$
is an equivalence of categories.
 We have the commutative diagram
$$\xymatrix{ {_R\M} \ar[rr]^{K\quad} \ar[drr]_{C \ot_R-} & & {^{\bC^e}\M}
                        \ar[d]^{^{\bC^e}\!U} \\
 & & {_R\M} \, . }
$$
\end{thm}

By Proposition \ref{P.1.7}, the functor $K$ is an equivalence provided
\begin{rlist}
  \item the functor $C\ot_R-:{_R\M}\to {_R\M}$ is comonadic, and
  \item the induced comonad morphism
  $C\ot_R \Hom_R(C,-) \to {\bC^e} \ot_R -$\\
  is an isomorphism.

If $R\simeq \End^{\bC^e}(C)\simeq Z(C^*)$,
the morphism in (ii) characterises $C$ as a {\em $\bC^e$-Galois comodule} as defined
in \cite[4.1]{W} and
if $C_R$ is finitely generated and projective, the condition reduces to
 an $R$-coalgebra isomorphism $C  \ot_R C^* \simeq {C^e}$.
\end{rlist}

In module categories, separable coalgebras are well studied and we
recall some of their characterisations
(e.g. Section \ref{adj-sep}, \cite{Lar}, \cite{Guz}, \cite[3.29]{BW}, \cite[2.10]{BBW}).

\begin{thm}\label{cosep}{\bf Coseparable coalgebras.} \em An $R$-coalgebra
$\bC=(C, \Delta, \ve)$ is called
{\em coseparable} if any of the following equivalent conditions is satisfied:
\begin{blist}
\item $C\ot_R-:\M_R\to \M_R$ is a separable comonad;
\item $\Delta: C\to C\ot_RC$ splits in ${^{\bC^e}\M}$;
\item $C$ is $(\bC^e, R)$-injective;
\item  the forgetful functor $^C\M\to \M_R$ is separable;
\item  the forgetful functor $^{C^e}\M\to \M_R$ is separable;
\item $\Hom_R(C,-): \M_R\to \M_R$ is a separable monad.
\end{blist}
\smallskip

{\em
For any coseparable coalgebra $\bC$,  $Z(C^*)$ is a direct summand of $  C^*$.}
\medskip

\begin{proof}  Let $\omega: C\ot_R C\to C$ denote the splitting morphism for $\Delta$.
Then we obtain the splitting sequence of $Z(C^*) $-modules

\smallskip
\qquad  \qquad $ C^* \simeq \Hom^{\bC^e}(C, C\ot_R C)
  \xra{\Hom^{\bC^e}(C, \omega)}\Hom^{\bC^e}(C, C)\simeq Z(C^*).$
\end{proof}
\end{thm}

For an Azumaya coalgebra $\bC$, the free functor
$\phi_{(\bC^\tau)_l}:\M_R  \to {^{\bC^\tau}\M}$
is monadic (see Theorem \ref{th.3}), and hence, in particular, it is conservative.
It then follows that, for each $X \in \M_R$, the morphism
$\ve \ot_R X:C \otimes_R X \to X$
 is surjective. For $X=R$ this yields that
$\ve:C  \to R $ is surjective (hence splitting). By Theorem
\ref{Azu-mon-comon} this means that $\bC$ is also a coseparable coalgebra.

It follows from the general Hom-tensor relations that the functor $K:\M_R \to {^{\bC^e}\M}$
has a right adjoint ${^{\bC^e}\Hom}(C,-):{^{\bC^e}\M}\to \M_R$ (e.g. \cite[3.9]{BW})
and we denote the unit and counit of this adjunction by
$\underline\eta$ and $\underline\ve$, respectively.

Besides the characterisations derived from Theorem \ref{th.1.Clg} we have:

\begin{thm}\label{char-azu-co} {\bf Characterisation of Azumaya coalgebras.}
For an $R$-coalgebra $\bC$ the following are equivalent:
\begin{blist}
\item $\bC$ is an Azumaya coalgebra;
\item
\begin{rlist}
\item $\underline \ve_X: C\ot_R {^{\bC^e}\Hom}(C,X)\to X$
     is an isomorphism for any $X\in {^{\bC^e}\M}$,
\item $\underline\eta_M: M \mapsto {^{\bC^e}\Hom}(C,C \ot_R M) $
 is an isomorphism for any $M\in \M_R$.
\end{rlist}
\item
  $C$ is a  $\bC^e$-Galois comodule,
  $C^*$ is a central $R$-algebra,
   and the functor
 $C \ot_R-:{_R\M} \to {_R\M}$ is comonadic;
\item $C^*$ is an Azumaya algebra.
\end{blist}
\end{thm}
\begin{proof} This is essentially Theorem \ref{Azu-mon-comon}.
\end{proof}

As shown in Proposition \ref{finite.d}, an Azumaya coalgebra $\bC$ is finite
in $\M_R$, that is, $C_R$ is finitely generated and projective (see Remark \ref{f.g.p}).

  Coalgebras $\bC$ with $C_R$ finitely generated and projective for which $C^*$
is an Azumaya $R$-algebra were investigated by K. Sugano in \cite{Sugano}.
As an easy consequence he also observed
that  an  $R$-algebra $\bA$ with $A_R$ finitely generated and projective
is Azumaya if and only if $A^*$ is an Azumaya coalgebra.
\smallskip

For vector space categories,
Azumaya $\bD$-coalgebras $\bC$ over a cocommutative coalgebra $\bD$ (over a field)
were defined and characterised in \cite[Theorem 3.14]{Tor-Brauer}.

\bigskip

{\bf Acknowledgments.} This research was partially supported by
Volkswagen Foundation (Ref.: I/85989).
The first author also gratefully acknowledges the support of the
Shota Rustaveli National Science Foundation Grant DI/12/5-103/11.

\smallskip

\noindent
{\bf Addresses:} \\[+1mm]
{A. Razmadze Mathematical Institute
of I. Javakhishvili Tbilisi State University,\\
6, Tamarashvili Str.,  {\small and} \\
 {Tbilisi Centre for Mathematical Sciences,
Chavchavadze Ave. 75, 3/35, \\
Tbilisi 0177},  Georgia.\\
    {\small bachi@rmi.ge}\\[+1mm]
{Department of Mathematics of HHU, 40225 D\"usseldorf, Germany},\\
  {\small wisbauer@math.uni-duesseldorf.de}

\end{document}